\documentclass{article}
\usepackage{graphicx} 
\usepackage{bbm}
\usepackage{multirow}%
\usepackage{mathrsfs}
\usepackage{subcaption} 
\usepackage{float}     
\usepackage{amsmath}
\usepackage{amssymb, amsfonts, amsmath, nicefrac,color}
\usepackage{epsfig}
\usepackage{dsfont}
\usepackage{bm}
\usepackage{color,xcolor}
\usepackage{algorithm}%
\usepackage{algorithmicx}%
\usepackage{algpseudocode}%

\def\prob {{\rm Prob}}

\newtheorem{theorem}{Theorem}[section]
\newtheorem{lemma}{Lemma}[section]
\newtheorem{proposition}{Proposition}[section]
\newtheorem{example}{Example}[section]
\newtheorem{remark}{Remark}[section]
\newtheorem{definition}{Definition}[section]
\newtheorem{assumption}{Assumption}[section]

\newcommand{\setd}{{ d \kern -.15em l}}
\newcommand{\hatsetd}{ d \hat{\kern -.15em l }}
\DeclareMathOperator*{\argmax}{arg\,max}
\DeclareMathOperator*{\argmin}{arg\,min}

\newcommand{\y}{{ y}}

\newcommand{\bgeqn}{\begin{eqnarray}}
\newcommand{\edeqn}{\end{eqnarray}}
\newcommand{\bgeq}{\begin{eqnarray*}}
\newcommand{\edeq}{\end{eqnarray*}}
\newcommand{\bec}{\begin{center}}
\newcommand{\enc}{\end{center}}

\newcommand{\inmat}[1]{\mbox{\rm {#1}}}

\numberwithin{equation}{section}

\newcommand{\F}{{\cal F}}

\newcommand{\B}{{\cal B}}

\newcommand{\be}{\begin{equation}}
\newcommand{\ee}{\end{equation}}

\def\dist{\mathop{\rm dist}}
\def\w{\omega}

\def\dom{{\rm dom}}

\def\bbe{{\Bbb{E}}} 

\title{Solutions of Two-stage  Stochastic Minimax Problems}
\author{Hailin Sun\footnote{School of Mathematical Sciences,  Nanjing Normal University, Nanjing, 210023, China, 
             {hlsun@njnu.edu.cn} }         \and
        Xiaojun Chen \footnote{  Corresponding author. Department of Applied Mathematics, The Hong Kong Polytechnic University, Hong Kong, China,  {xiaojun.chen@polyu.edu.hk}}
}

\date{\today}

\begin{document}

\maketitle

\begin{abstract}
This paper introduces a class of two-stage stochastic minimax problems where the first-stage objective function is nonconvex-concave while the second-stage objective function is strongly convex-concave. We establish properties of the second-stage minimax value function and solution functions, and characterize the existence and relationships among saddle points, minimax points, and KKT points.  We apply the sample average approximation (SAA) to the class of two-stage stochastic minimax problems and prove the convergence of the KKT points as the sample size tends to infinity. An inexact parallel proximal gradient descent ascent  algorithm is proposed to solve this class of problems with the SAA. Numerical experiments demonstrate the effectiveness of the proposed algorithm and
validate the convergence properties of the SAA approach.
\end{abstract}



\section{Introduction}
\label{intro}
In this paper, we consider the following two-stage stochastic minimax problem
\begin{equation}\label{eq:ts-minimax}
\min_{x_1\in X_1} \max_{y_1\in Y_1} \;  \psi(x_1, y_1): =F_1(x_1, y_1) + \bbe\left[\psi_2(x_1, y_1, \xi)\right],
\end{equation}
where
\begin{equation}\label{eq:st-minimax}
\psi_2(x_1, y_1, \xi):=\min_{x_2\in X_2(x_1, \xi)}\max_{y_2\in Y_2(y_1, \xi)}F_2(x_2, y_2, \xi),
\end{equation}
$\mathbb{E}$ denotes the expectation, $\xi:\Omega\to\Xi\subset\mathbb{R}^l$ is a random variable defined on the probability space $(\Omega, \F, P)$, $X_1\subset\mathbb{R}^{n_1}$, $Y_1\subset\mathbb{R}^{m_1}$ are  convex compact sets, $F_1:\mathbb{R}^{n_1}\times\mathbb{R}^{m_1}\to\mathbb{R}$, $F_2:\mathbb{R}^{n_2}\times\mathbb{R}^{m_2}\times\mathbb{R}^{l}\to\mathbb{R}$,
$$
X_2(x_1, \xi):=\{x_2\in \mathbb{R}^{n_2} : T(\xi)x_1+W(\xi)x_2\leq h(\xi) \},
$$
$$
Y_2(y_1, \xi):=\{y_2\in  \mathbb{R}^{m_2} : A(\xi)y_1+B(\xi)y_2\leq c(\xi) \},
$$
and $T(\xi)\in\mathbb{R}^{l_2\times n_1}$, $W(\xi)\in\mathbb{R}^{l_2\times n_2}$, $A(\xi)\in\mathbb{R}^{s_2\times m_1}$, $B(\xi)\in\mathbb{R}^{s_2\times m_2}$, $h(\xi)\in\mathbb{R}^{l_2}, c(\xi)\in\mathbb{R}^{s_2}$ for all $\xi\in\Xi$.
The objective functions in \eqref{eq:ts-minimax} and \eqref{eq:st-minimax} have the following structure.
\begin{itemize}
\item $F_1(x_1, y_1):= f(x_1) + \psi_1(x_1, y_1) - g(y_1)$, where
 $f:\mathbb{R}^{n_1}\to\mathbb{R}$ is a proper and lower semicontinuous (lsc) function,
 $\psi_1:\mathbb{R}^{n_1}\times\mathbb{R}^{m_1}\to\mathbb{R}$ is a Lipschitz continuously differentiable function,
 $g:\mathbb{R}^{m_1}\to\mathbb{R}$ is a proper and lsc function, and for every $x_1\in X_1$, $g(\cdot)- \psi_1(x_1, \cdot)$ is $\sigma$-strongly convex with $\sigma>0$.

    \item
    For almost every (a.e.) $\xi\in\Xi$, $F_2(\cdot, \cdot, \xi)$ is a $\sigma$-strongly convex-strongly concave function.
\end{itemize}

The two-stage stochastic minimax problem  \eqref{eq:ts-minimax}-\eqref{eq:st-minimax} represents an extension of the classical two-stage stochastic minimization model originally developed by Dantzig \cite{dantzig1955linear} and Beale \cite{beale1955minimizing}.
Two-stage stochastic minimization problem \cite{birge1997introduction,shapiro2021lectures} is a fundamental framework for sequential decision-making under uncertainty, where first-stage ``here-and-now" decisions (e.g., infrastructure investments) must be made before observing random outcomes, followed by second-stage ``wait-and-see" recourse actions (e.g., operational adjustments) that adapt to the realized uncertainty. Due to its modeling power, the two-stage stochastic minimization problem has been widely used in many important engineering and scientific applications, such as wireless resource optimization \cite{liu2021two}, transportation network design \cite{liu2009two}, and machine/deep learning \cite{lee2018resource}.

In contrast to stochastic optimization's expectation-based framework, the minimax approach explicitly considers worst-case scenarios to achieve robust solutions. The study of minimax problems can be traced back to von Neumann's seminal work \cite{v1928theorie} in 1928  on convex-concave deterministic minimax problems.
In recent years, nonconvex-nonsmooth minimax problems \cite{bian2024nonsmooth,cohen2025alternating,jiang2023optimality,jin2020local,xu2023unified} and stochastic minimax problems \cite{chen2024near,lan2023novel,shapiro2002minimax} have garnered significant attention due to their applications in data science, machine learning, game theory, and robust decision-making.  Jin et al.  \cite{jin2020local} gave the definitions of global minimax points and local minimax points by considering the minimax problem as a two-player sequential game.
Cohen and  Teboulle \cite{cohen2025alternating} analyzed proximal gradient methods for nonconvex and strongly concave minimax problems. 
Chen and Luo \cite{chen2024near} proposed a recursive anchored iteration method for smooth minimax problems and proved that their algorithm achieves near-stationarity.

To the best of our knowledge,  the two-stage stochastic minimax problem has not been investigated, although two-stage stochastic minimization and minimax problem have been extensively studied as two separate classes of mathematical models.
However, many practical scenarios involve stochastic environments that combine both risk-neutral and risk-averse components, requiring decision-makers to simultaneously address:
(i) sequential decision-making,
(ii) stochastic uncertainties, and
(iii) adversarial considerations (extreme events or opponent behaviors).
This constitutes the primary motivation for studying two-stage stochastic minimax problems (see Example \ref{ex:1}).

Another motivation for studying the two-stage stochastic minimax problem \eqref{eq:ts-minimax}-\eqref{eq:st-minimax} lies in its capacity to model two-stage stochastic two-player zero-sum games (see Example \ref{ex:2}), which constitute a specialized subclass of two-stage stochastic Nash equilibrium problems (SNEPs).
Pang et al. \cite{pang2017two} examined two-stage SNEPs involving risk-averse players under uncertainty, developing an iterative best-response solution framework. Zhang et al. \cite{zhang2019two} subsequently investigated a mixed non-cooperative game formulation for two-stage decision-making in uncertain environments. Further advancing this line of research, Lei et al. \cite{lei2020synchronous} introduced synchronous, asynchronous, and randomized best-response schemes for SNEPs, with specific applications to two-stage SNEPs featuring both linear and quadratic recourse structures.

The KKT conditions for two-stage SNEPs can be characterized through two-stage stochastic variational inequalities (SVIs).  Rockafellar and Wets \cite{rockafellar2017stochastic} and Chen et al. \cite{chen2017two} extended SVIs from  single-stage to multi-stage and  two-stage SVIs, respectively. Rockafellar and Sun \cite{rockafellar2019solving} studied the progressive hedging algorithm (PHA) for solving multi-stage SVIs when the random variable follows a discrete distribution. Discrete approximation methods have been proposed to approximate two-stage SVIs  \cite{chen2019discrete,chen2019convergence}, allowing the approximated SVIs to be solved using PHA. In addition to PHA, dynamic stochastic approximation-type algorithms effectively solve  two-stage SVIs \cite{chen2022stochastic}.

 However,  most existing research on two-stage SNEPs and two-stage SVIs has been confined to smooth and monotone cases, leaving further developments for nonsmooth and nonconvex problems  unexplored.



The main contributions of this paper are summarized as follows.

\begin{itemize}

\item We introduce a  two-stage stochastic minimax model \eqref{eq:ts-minimax}-\eqref{eq:st-minimax} and investigate the properties of the second-stage minimax value functions and solution functions. Based on these properties, we examine the existence and relationships among saddle points, minimax points, and KKT points for the nonconvex-nonsmooth two-stage stochastic minimax problem \eqref{eq:ts-minimax}-\eqref{eq:st-minimax}.

\item We apply the sample average approximation (SAA) method to problem \eqref{eq:ts-minimax}-\eqref{eq:st-minimax} and  prove that the divergence between the KKT point sets of the SAA problem and the true problem \eqref{eq:ts-minimax}-\eqref{eq:st-minimax} converges to zero almost surely.

\item We propose an Inexact Parallel Proximal Gradient Descent Ascent (IPPGDA) algorithm for solving problem \eqref{eq:ts-minimax}-\eqref{eq:st-minimax}, with both subsequence and global convergence analyses. Unlike the single-stage minimax problems solved by the parallel proximal gradient descent-ascent (PPGDA) algorithm in \cite{cohen2025alternating}, the inexactness in our approach arises not only from solving the inner maximization problem but also from solving the second-stage minimax problem.

\end{itemize}

The remainder of this paper is organized as follows. Section 2 introduces the motivating examples, necessary assumptions, and fundamental concepts. Section 3 investigates the properties of  problem \eqref{eq:ts-minimax}-\eqref{eq:st-minimax}, analyzing the existence and relationships among saddle points, minimax points and KKT points. Section 4 examines the SAA method for problem \eqref{eq:ts-minimax}-\eqref{eq:st-minimax}, including its convergence properties. In section 5, we present IPPGDA algorithm for solving problem \eqref{eq:ts-minimax}-\eqref{eq:st-minimax}, along with its subsequence and global convergence analysis, and demonstrate the effectiveness of our proposed algorithm and the convergence behavior of the SAA method through preliminary numerical experiments. Section 6 concludes the paper. 

{\bf Notation} For a  Lipschitz continuous function $f:\mathbb{R}^n\to\mathbb{R}$, $\partial f(\bar{x})$ denotes the Clarke subdifferential \cite{clarke1990optimization} of $f$  at point $\bar{x}$. For a  Lipschitz continuously differentiable function $g:\mathbb{R}^n\times\mathbb{R}^m\to\mathbb{R}$, $\partial^2g(\bar{x}, \bar{y})$ represents the Clarke generalized Hessian of $g$ at point $(\bar{x}, \bar{y})$, while $\partial_{xx}^2g(\bar{x}, \bar{y})$ and $\partial_{yy}^2g(\bar{x}, \bar{y})$ denote the Clarke generalized Hessians of $g$  with respect to (w.r.t.) $x$ and $y$ at point $(\bar{x}, \bar{y})$.
For a vector-valued function $H:\mathbb{R}^n\times\mathbb{R}^m\to\mathbb{R}^n$, $\partial_x H(\bar{x},\bar{y})$ denotes the Clarke generalized Jacobian of $H$ w.r.t. $x$ at point $(\bar{x},\bar{y})$.
A function $g:\mathbb{R}^n\times\mathbb{R}^m\to\mathbb{R}$ is said (strongly) convex-(strongly) concave if $g(\cdot, y)$ is (strongly) convex in $\mathbb{R}^n$ for any fixed $y\in\mathbb{R}^m$, and $g(x, \cdot)$ is (strongly) concave in $\mathbb{R}^m$ for any fixed $x\in\mathbb{R}^n$.
For $a\in\mathbb{R}^n$, $\|a_+\|_0:=\sum_{i=1}^n(\max\{a_i, 0\})^0$ with $0^0=0$, where $a_i$ is the $i$th element of $a$. For $x\in\mathbb{R}^n$ and $X, Y\subseteq\mathbb{R}^n$, $\mathbb{D}(x, Y):=\inf_{y\in Y}\|x-y\|$ and $\mathbb{D}(X, Y):=\sup_{x\in X}\inf_{y\in Y}\|x-y\|$. Let  $\mathcal{L}^{n}_p$ denote the Lebesgue space of measurable functions from a measure space to $\mathbb{R}^{n}$ with finite $L^p$-norm, where $p\geq1$ and for any measurable function ${\bf x}\in \mathcal{L}^{n}_p$, the $L^p$-norm is defined as $\|{\bf x}\|_p:=(\int_{\Omega}\sum_{i=1}^n|({\bf x}(\xi(\w)))_i|^pP(d\w))^{1/p}$. For a convex set $X\subset\mathbb{R}^n$ and $x\in \mathbb{R}^n$, $\mathcal{N}_X(x)$ denotes the normal cone to $X$ at $x$.

\section{Examples, assumptions and concepts}
\label{sec:1}
In this section, we provide two motivating examples for the study of the two-stage stochastic minimax problem \eqref{eq:ts-minimax}-\eqref{eq:st-minimax}. Additionally, we introduce several necessary assumptions and concepts required for the theoretical analysis in this paper.

The first example is a renewable energy storage scheduling problem in a stochastic environment that combines both risk-neutral and risk-averse criteria.

\begin{example}\label{ex:1}[{Renewable energy storage scheduling}]:
A microgrid operator selects energy storage capacity $x \geq 0$ in the first stage to minimize the total cost, which comprises two components:
(1) The investment cost $C_{\text{inv}}(x, z)$, subject to an adversarial uncertainty $z\in Z$ representing risks like volatile material prices or uncertain government subsidies;
(2) The expected future losses under renewable generation uncertainty $\xi$ and adversarial electricity prices $p \in \mathcal{P}(\xi)\subset\mathbb{R}^{24}_+$, where the uncertainty set expands price bounds proportionally to the renewable forecast error $\xi$.

In the second stage, after observing $\xi$, the operator adjusts charge/discharge decisions $y_t^{\text{ch}}$ and $y_t^{\text{dis}}$ ($t = 1,\dots,24$) subject to storage constraints: power limits $0 \leq y_t^{\text{ch}}, y_t^{\text{dis}} \leq 0.2x$, state-of-charge dynamics $\text{SOC}_t = \text{SOC}_{t-1} + 0.9y_t^{\text{ch}} - y_t^{\text{dis}}/0.9$ with $0 \leq \text{SOC}_t \leq x$, and boundary conditions $\text{SOC}_0 = \text{SOC}_{24} = 0$. We denote by $Y(x)$ the feasible region of $y:=(y_t^{\text{ch}}, y_t^{\text{dis}})_{t = 1}^{24}$ satisfying above constraints.
Moreover,
the second-stage objective function is
\[
Q(y, p, \xi): =  \sum\limits_{t} \Big[ p_t(y_t^{\text{dis}} - y_t^{\text{ch}}) + 0.1(y_t^{\text{ch}} + y_t^{\text{dis}}) \Big] + 100 \left( \xi + \sum\limits_t (y_t^{\text{dis}} - y_t^{\text{ch}}) \right)^2,
\]
which captures real-time market costs, battery degradation penalties, and renewable energy imbalance penalties. The complete formulation, which integrates both stages, is a two-stage stochastic minimax problem as follows:
$$
\min_{x \geq 0}\max_{z\in Z} \left( C_{\text{inv}}(x,z) + \mathbb{E}_{\xi} \left[ \min\limits_{y\in Y(x)}\max_{p \in \mathcal{P}(\xi)} Q(y, p, \xi) \right] \right).
$$
\end{example}

The second example concerns a two-stage stochastic two-player zero-sum game.

\begin{example}\label{ex:2}[{Two-stage stochastic two-player zero-sum game}]: The two-player zero-sum game is a basic model in game theory \cite{washburn2014two}. There are two players, each with an associated set of strategies. While one player aims to maximize her payoff, the other player attempts to take action to minimize this payoff. 

The two-stage stochastic two-player zero-sum game is a two-stage stochastic extension of the two-player zero-sum game. There are two players in the game, player 1's problem is
\begin{equation}\label{eq:tzs-game1}
\min_{x_1\in X_1} \; F_1(x_1, y_1) + \bbe\left[\min_{x_2\in X_2(x_1, \xi)}F_2(x_2, y_2, \xi)\right]
\end{equation}
and player 2's problem is
\begin{equation}\label{eq:tzs-game2}
\min_{y_1\in Y_1} \; -F_1(x_1, y_1) + \bbe\left[\min_{y_2\in Y_2(y_1, \xi)} -F_2(x_2, y_2, \xi)\right].
\end{equation}
When $F_1$ and $F_2$ are convex w.r.t. $x_1$ and $x_2$, respectively, and concave w.r.t. $y_1$ and $y_2$, respectively, and when $X_1$, $Y_1$, $X_2(x_1, \xi)$ and $Y_2(y_1, \xi)$ are convex and compact sets, then \eqref{eq:tzs-game1}-\eqref{eq:tzs-game2} is equivalent to problem \eqref{eq:ts-minimax}-\eqref{eq:st-minimax}.

\end{example}

To investigate the two-stage stochastic minimax problem \eqref{eq:ts-minimax}-\eqref{eq:st-minimax}, we need the following assumptions.

\begin{assumption}\label{a:second-stage}
For every $(x_1, y_1)\in X_1\times Y_1$ and
$\xi\in\Xi$,
\begin{itemize}
\item[(i)] $F_2$ is continuous, and $F_2(\cdot, \cdot, \xi)$ is Lipschitz continuously differentiable. 
Moreover,
 $\nabla_{x_2} F_2$ and $\nabla_{y_2} F_2$ are Lipschitz continuous;

\item[(ii)] $X_2(x_1, \xi)$ and $Y_2(y_1, \xi)$ are nonempty;



\item[(iii)] 
$B(\xi)$ and $W(\xi)$ are of full row rank.
\end{itemize}
\end{assumption}


\begin{assumption}\label{a:xicompact-regular}
The support set $\Xi\subset\mathbb{R}^l$ of the random vector $\xi$ is compact. The random matrices $A(\cdot)$, $B(\cdot)$, $T(\cdot)$, $W(\cdot)$ and random vectors $c(\cdot), h(\cdot)$ are continuous. 
\end{assumption}

\begin{remark}
In the case when $\Xi$ is unbounded, under the tightness of random variable $\xi$, for any $\epsilon\in(0,1)$ there exists a compact subset $\bar{\Xi}\subset \Xi$ such that $\prob\{\xi\in\bar{\Xi}\}\geq 1-\epsilon$. Then we may omit the $\xi\in\Xi\backslash\bar{\Xi}$ and consider the compact support $\bar{\Xi}$ in our problem.
\end{remark}

We  provide the definitions of the following concepts for the two-stage stochastic minimax problem \eqref{eq:ts-minimax}-\eqref{eq:st-minimax}: saddle point, local saddle point, global minimax point, and local minimax point.

\begin{definition}
A point $(x_1^*, y_1^*)\in X_1\times Y_1$ is called a saddle point of problem \eqref{eq:ts-minimax}-\eqref{eq:st-minimax} if for all $(x_1, y_1)\in X_1\times Y_1$, it holds
\begin{equation}\label{eq:saddle21}
\psi(x_1^*, y_1)\leq \psi(x_1^*, y_1^*)\leq \psi(x_1, y_1^*).
\end{equation}
We call $(x_1^*, y_1^*)\in X_1\times Y_1$ a local saddle point of problem \eqref{eq:ts-minimax}-\eqref{eq:st-minimax}, if there exists a $\delta>0$ such that \eqref{eq:saddle21} holds for all $(x_1, y_1)\in (X_1\times Y_1)\cap \B((x_1^*, y_1^*), \delta)$, where $\B((x_1^*, y_1^*), \delta)$ denotes the neighborhood of $(x_1^*, y_1^*)$ with radius $\delta$.
\end{definition}

\begin{definition}\label{d:saddle-mapping-2}
A pair of mappings $(\bar{{\bf x}}_2: X_1\times Y_1\times \Xi\to\mathbb{R}^{n_2}, \bar{{\bf y}}_2: X_1\times Y_1\times \Xi\to\mathbb{R}^{m_2})$ is called a  saddle point mapping of the second-stage minimax problem in \eqref{eq:st-minimax}, if $(x_2^*, y_2^*):=(\bar{{\bf x}}_2(x_1, y_1, \xi), \bar{{\bf y}}_2(x_1, y_1, \xi))$ is a saddle point of the minimax problem in \eqref{eq:st-minimax}, namely, for any $y_2\in Y_2(y_1, \xi)$ and $x_2\in X_2(x_1, \xi)$,
\[
F_2(x_2^*, y_2, \xi)\leq F_2(x_2^*, y_2^*, \xi)\leq F_2(x_2, y_2^*, \xi).
\]
\end{definition}


\begin{definition}\label{d:glolocalminimax}
A point $(\bar{x}_1, \bar{y}_1)\in X_1\times Y_1$ is a global minimax point of problem \eqref{eq:ts-minimax}-\eqref{eq:st-minimax}, if
\begin{equation}\label{eq:fs-global}
\psi(\bar{x}_1, y_1) \leq \psi(\bar{x}_1, \bar{y}_1) \leq \max_{y_1'\in Y_1} \psi(x_1, y_1')
\end{equation}
holds for any $(x_1, y_1)\in X_1\times Y_1$. 

Moreover,
a point $(\bar{x}_1, \bar{y}_1)\in X_1\times Y_1$ is a local minimax point of problem \eqref{eq:ts-minimax}-\eqref{eq:st-minimax}, if there exists $\delta_0>0$, such that for any $\delta\in(0, \delta_0]$ and any $(x_1, y_1)\in X_1\times Y_1$ satisfying $\|x_1-\bar{x}_1\|\leq \delta$ and $\|y_1-\bar{y}_1\|\leq \delta$, we have
\[
\psi(\bar{x}_1, y_1) \leq \psi(\bar{x}_1, \bar{y}_1) \leq \max_{y_1'\in Y_1} \psi(x_1, y_1').
\]
\end{definition}

Let $\bar{\psi}(x_1):=\max_{y_1\in Y_1}\psi(x_1, y_1)$. A point $x_1^*$ is called a stationary point of
\begin{equation}\label{eq:first-stage-min}
\min_{x_1\in X_1} \; \bar{\psi}(x_1),
\end{equation}
if $0\in \partial \bar{\psi}(x_1^*)+\mathcal{N}_{X_1}(x_1^*)$.

\begin{remark}\label{r:localminimax}
Note that the definition of a local minimax point in Definition \ref{d:glolocalminimax} implies that $\bar{y}_1$ is a local maximum point of $\psi(\bar{x}_1, \cdot)$ (since $\psi(\bar{x}_1, \cdot)$ is strongly concave, $\bar{y}_1$ is also a global maximum point of $\psi(\bar{x}_1, \cdot)$), and $\bar{x}_1$ is a local minimum of $\bar{\psi}(\cdot)$.
Note also that this definition is a simplified version of \cite[Definition 14]{jin2020local}, due to the strong concavity of $\psi$ in $y_1$.
\end{remark}

We also consider the following minimax problem
\begin{equation}\label{eq:ts-minimax-interchange}
\min_{(x_1,  {\bf x}_2)\in {\bf X}} \max_{(y_1, {\bf y}_2)\in {\bf Y}} \; F_1(x_1, y_1) + \bbe\left[F_2({\bf x}_2(\xi), {\bf y}_2(\xi), \xi)\right],
\end{equation}
where
\begin{multline*}
{\bf X}:=\{(x_1, {\bf x}_2)\in\mathbb{R}^{n_1}\times\mathcal{L}^{n_2}_p: x_1\in X_1,
T(\xi)x_1+W(\xi) {\bf x}_2(\xi)\leq h(\xi), \;a.e.\; \xi\in\Xi\}
\end{multline*}
and
\begin{multline*}
{\bf Y}:=\{(y_1, {\bf y}_2)\in\mathbb{R}^{m_1}\times\mathcal{L}^{m_2}_p: y_1\in Y_1,
A(\xi)y_1+B(\xi){\bf y}_2(\xi)\leq c(\xi), \;a.e.\; \xi\in\Xi\}.
\end{multline*}
We will consider the relationship between \eqref{eq:ts-minimax-interchange} and the two-stage stochastic minimax problem \eqref{eq:ts-minimax}-\eqref{eq:st-minimax} in Section \ref{sec:31}. Moreover, the definition of a saddle point, local saddle point, global minimax point and local minimax point for the minimax problem \eqref{eq:ts-minimax-interchange} is as follows.

\begin{definition}
A point $(x_1^*, {\bf x}_2^*,y_1^*, {\bf y}_2^*)\in {\bf X}\times {\bf Y}$ is called a   saddle point of the minimax problem \eqref{eq:ts-minimax-interchange} if for any $(x_1,  {\bf x}_2, y_1, {\bf y}_2)\in {\bf X}\times {\bf Y}$, we have
\begin{equation}\label{eq:twostagesaddlepoint}
\begin{array}{lll}
F_1(x_1^*, y_1) + \bbe\left[F_2({\bf x}_2^*(\xi), {\bf y}_2(\xi), \xi)\right]
&\leq& F_1(x_1^*, y_1^*) + \bbe\left[F_2({\bf x}_2^*(\xi), {\bf y}_2^*(\xi), \xi)\right]\\
&\leq& F_1(x_1, y_1^*) + \bbe\left[F_2({\bf x}_2(\xi), {\bf y}_2^*(\xi), \xi)\right].
\end{array}
\end{equation}
 We call $(x_1^*, {\bf x}_2^*,y_1^*, {\bf y}_2^*)\in {\bf X}\times {\bf Y}$ a   local saddle point of the minimax problem \eqref{eq:ts-minimax-interchange} if there exists a $\delta>0$ such that \eqref{eq:twostagesaddlepoint} holds for all $(x_1, {\bf x}_2, y_1, {\bf y}_2)\in ({\bf X}\times {\bf Y})\cap\B((x_1^*, {\bf x}_2^*,y_1^*, {\bf y}_2^*), \delta)$, where $\B((x_1^*, {\bf x}_2^*,y_1^*, {\bf y}_2^*), \delta)$ denotes the neighborhood of $(x_1^*, {\bf x}_2^*,y_1^*, {\bf y}_2^*)$ with radius $\delta$, that is $\B((x_1^*, {\bf x}_2^*,y_1^*, {\bf y}_2^*), \delta)=\{(x_1, {\bf x}_2,y_1, {\bf y}_2)\in\mathbb{R}^{n_1}\times\mathcal{L}^{n_2}_p\times\mathbb{R}^{m_1}\times\mathcal{L}^{m_2}_p: \|x_1-x_1^*\|_2+\|{\bf x}_2-{\bf x}_2^*\|_p+\|y_1-y_1^*\|_2+\|{\bf y}_2-{\bf y}_2^*\|_p\leq \delta\}$.
\end{definition}

\begin{definition}
A point $(\tilde{x}_1, \tilde{{\bf x}}_2,\tilde{y}_1, \tilde{{\bf y}}_2)\in {\bf X}\times {\bf Y}$ is called a global minimax point of the minimax problem \eqref{eq:ts-minimax-interchange} if
\begin{align}
F_1(\tilde{x}_1, y_1) + \mathbb{E}\left[F_2(\tilde{{\bf x}}_2(\xi), {\bf y}_2(\xi), \xi)\right] &\leq F_1(\tilde{x}_1, \tilde{y}_1) + \mathbb{E}\left[F_2(\tilde{{\bf x}}_2(\xi), \tilde{{\bf y}}_2(\xi), \xi)\right] \nonumber\\
\leq &\max_{(y'_1, {\bf y}'_2(\xi))\in {\bf Y}}F_1(x_1, y'_1) + \mathbb{E}\left[F_2({\bf x}_2(\xi), {\bf y}'_2(\xi), \xi)\right]\label{eq:ts-global}
\end{align}
holds for any $(x_1, {\bf x}_2, y_1, {\bf y}_2)\in {\bf X}\times {\bf Y}$.

Moreover, a point $(\tilde{x}_1, \tilde{{\bf x}}_2,\tilde{y}_1, \tilde{{\bf y}}_2)\in {\bf X}\times {\bf Y}$ is called a local minimax point of the minimax problem \eqref{eq:ts-minimax-interchange}, if there exists $\delta_0>0$,
such that for any $\delta\in(0, \delta_0]$ and any $(x_1, {\bf x}_2, y_1, {\bf y}_2)\in {\bf X}\times {\bf Y}$ satisfying $\|x_1-\tilde{x}_1\| +\| {\bf x}_2- \tilde{{\bf x}}_2\|_p\leq \delta$ and $\|y_1 -\tilde{y}_1\| +\| {\bf y}_2- \tilde{{\bf y}}_2\|_p\leq \delta$, we have
\begin{align*}
F_1(\tilde{x}_1, y_1) + \mathbb{E}\left[F_2(\tilde{{\bf x}}_2(\xi), {\bf y}_2(\xi), \xi)\right] &\leq F_1(\tilde{x}_1, \tilde{y}_1) + \mathbb{E}\left[F_2(\tilde{{\bf x}}_2(\xi), \tilde{{\bf y}}_2(\xi), \xi)\right] \\
\leq &\max_{(y'_1, {\bf y}'_2(\xi))\in {\bf Y}}F_1(x_1, y_1') + \mathbb{E}\left[F_2({\bf x}_2(\xi), {\bf y}_2'(\xi), \xi)\right].
\end{align*}
\end{definition}

\section{Properties of two-stage stochastic minimax problems}

In this section, 
we investigate the properties of minimax value functions, saddle points and minimax points of problem \eqref{eq:ts-minimax}-\eqref{eq:st-minimax}.

For given $\xi\in\Xi$ and  $(x_1, y_1)\in X_1\times Y_1$, we first investigate the properties of optimal value functions
\begin{equation}\label{eq:f21}
 f_{21}(x_2,  y_1, \xi):=\displaystyle{\max_{y_2\in Y_2(y_1, \xi)}F_2(x_2, y_2, \xi)},
\end{equation}
\begin{equation}\label{eq:f22}
 f_{22}(x_1, y_2,  \xi):=\displaystyle{\min_{x_2\in X_2(x_1, \xi)}F_2(x_2, y_2, \xi)},
\end{equation}
and minimax value function $\psi_2(x_1, y_1, \xi)$. Note that the minimax problem in \eqref{eq:st-minimax} is a strongly convex-strongly concave minimax problem with linear constraints for any $(x_1, y_1, \xi)$.
Then the KKT condition of the minimax problem in \eqref{eq:st-minimax} 
with given $(x_1, y_1)\in X_1\times Y_1$ can be stated as the following system of nonsmooth equations in the variable $\mu = (x_2, y_2, \pi_{x_2}, \pi_{y_2})$:
\begin{equation}\label{eq:nonsmooth-equation}
H(\mu,\xi) = \left(
\begin{array}{c}
 \nabla_{x_2}F_2(x_2, y_2, \xi) + W(\xi)^\top \pi_{x_2}\\
 -\nabla_{y_2}F_2(x_2, y_2, \xi) + B(\xi)^\top \pi_{y_2}\\
 \min( \pi_{x_2}, h(\xi)-T(\xi)x_1-W(\xi)x_2) \\
\min( \pi_{y_2}, c(\xi)-A(\xi)y_1-B(\xi)y_2)
\end{array}
\right)=0,
\end{equation}
where $\pi_{x_2}\in \mathbb{R}^{l_2}$ and $\pi_{y_2}\in \mathbb{R}^{s_2}$ are corresponding Lagrange multipliers, and  ``min" denotes the component-wise minimum operator on a pair of vectors.

\begin{lemma}\label{l:2stage-bound-continuity}
Under Assumptions \ref{a:second-stage}-\ref{a:xicompact-regular}, the following statements hold.
\begin{itemize}
\item[(i)] For any given $\xi\in\Xi$,
$f_{21}(\cdot, \cdot, \xi)$ is  continuously differentiable and strongly convex-concave; 
moreover,  $\nabla_{y_1}  f_{21}$ and  $\nabla_{x_2}  f_{21}$ are Lipschitz continuous w.r.t. $(x_2, y_1)$ and continuous w.r.t. $\xi$.

 \item[(ii)] For any given $\xi\in\Xi$,
 $f_{22}(\cdot, \cdot, \xi)$ is  continuously differentiable and  convex-strongly concave;
 moreover,  $\nabla_{x_1}  f_{22}$ and  $\nabla_{y_2}  f_{22}$ are Lipschitz continuous w.r.t. $(x_1, y_2)$ and continuous w.r.t. $\xi$.

\item[(iii)] For any given $\xi\in\Xi$, $\psi_2(\cdot, \cdot, \xi)$ is  convex-concave and continuously differentiable. And  there exist ${\bm \pi}_{x_2}: X_1\times Y_1\times \Xi\to\mathbb{R}^{l_2}$ and ${\bm \pi}_{y_2}: X_1\times Y_1\times \Xi\to\mathbb{R}^{s_2}$ such that
\[(\bar{{\bf x}}_2(x_1, y_1, \xi), \bar{{\bf y}}_2(x_1, y_1, \xi), {\bm \pi}_{x_2}(x_1, y_1, \xi), {\bm \pi}_{y_2}(x_1, y_1, \xi))\]
 satisfies the KKT condition 
 of the minimax problem in \eqref{eq:st-minimax}, where
\[
\nabla_{x_1} \psi_2(\cdot, \cdot, \xi)=T(\xi)^\top {\bm \pi}_{x_2}(\cdot, \cdot, \xi)\;,\;\nabla_{y_1} \psi_2(\cdot, \cdot, \xi)=-A(\xi)^\top {\bm \pi}_{y_2}(\cdot, \cdot, \xi),
\]
and $(\bar{{\bf x}}_2(x_1, y_1, \xi), \bar{{\bf y}}_2(x_1, y_1, \xi))$ is the unique saddle point of the minimax problem in \eqref{eq:st-minimax}.
Moreover, $(\bar{{\bf x}}_2, \bar{{\bf y}}_2,  {\bm \pi}_{x_2}, {\bm \pi}_{y_2})$ is Lipschitz continuous w.r.t. $(x_1, y_1)$ over $X_1\times Y_1$ and continuous w.r.t. $\xi$.

\item[(iv)]
 $(\bar{{\bf x}}_2(x_1, y_1, \xi), \bar{{\bf y}}_2(x_1, y_1, \xi))$ is contained in a convex and compact set $\hat{X}_2\times \hat{Y}_2$ for all $(x_1, y_1, \xi)\in X_1\times Y_1\times \Xi$. Moreover, there exist convex and compact sets $\check{X}_2$ and $\check{Y}_2$ such that $\hat{X}_2\subset\check{X}_2$, $\hat{Y}_2\subset\check{Y}_2$,
\begin{equation}\label{eq:f21check}
\displaystyle{\max_{y_2\in Y_2(y_1, \xi)\cap \check{Y}_2}F_2(x_2, y_2, \xi)}= f_{21}(x_2,  y_1, \xi),\;\; \forall (x_2, y_1, \xi)\in \hat{X}_2\times Y_1\times \Xi
\end{equation}
and
\begin{equation}\label{eq:f22check}
\hspace{-0.09in} \displaystyle{\min_{x_2\in X_2(x_1, \xi)\cap \check{X}_2}F_2(x_2, y_2, \xi)}= f_{22}(x_1,  y_2, \xi), \;\; \forall (x_1, y_2, \xi)\in X_1\times \hat{Y}_2\times \Xi.
\end{equation}
\end{itemize}
\end{lemma}

\begin{proof}
(i) For any given $(x_2, \xi)\in \mathbb{R}^{n_2}\times \Xi$, the  concavity of $ f_{21}$ w.r.t. $y_1$  is established in \cite[Proposition 2.21 (i)]{shapiro2021lectures}. Moreover, for any $\lambda\in(0,1)$ and $x_2^1, x_2^2\in \mathbb{R}^{n_2}$, we have
\[
\begin{array}{lll}
  &f_{21}(\lambda x_2^1 + (1-\lambda) x_2^2, y_1, \xi) \\=& \displaystyle{\max_{y_2\in Y_2(y_1, \xi)}} F_2(\lambda x_2^1 + (1-\lambda) x_2^2, y_2, \xi)\\
 \leq & \displaystyle{\max_{y_2\in Y_2(y_1, \xi)}}\lambda F_2( x_2^1, y_2, \xi) + (1-\lambda) F_2(x_2^2, y_2, \xi)-\frac{\sigma}{2}\lambda(1-\lambda)\|x_2^1-x_2^2\|^2\\
 \leq& \lambda \displaystyle{\max_{y_2\in Y_2(y_1, \xi)}} F_2( x_2^1, y_2, \xi) + (1-\lambda)\displaystyle{\max_{y_2\in Y_2(y_1, \xi)}} F_2(x_2^2, y_2, \xi)-\frac{\sigma}{2}\lambda(1-\lambda)\|x_2^1-x_2^2\|^2\\
 =& \lambda f_{21}( x_2^1 , y_1, \xi) + (1-\lambda)  f_{21}(x_2^2, y_1, \xi)-\frac{\sigma}{2}\lambda(1-\lambda)\|x_2^1-x_2^2\|^2,
\end{array}
\]
which implies the strong convexity of $ f_{21}$ w.r.t. $x_2$.

Note that the KKT condition of problem in \eqref{eq:f21} is
\begin{equation}\label{eq:kktf21-2}
    \begin{array}{l}
           0= -\nabla_{y_2} F_2(x_2, y_2, \xi) + B(\xi)^\top \pi_{y_2},  \\
         0\leq \pi_{y_2}\bot c(\xi)-A(\xi)y_1-B(\xi)y_2\geq 0,
    \end{array}
\end{equation}
where $\pi_{y_2}\in\mathbb{R}^{s_2}$ is the corresponding Lagrange multiplier. Let $(\hat{y}_2, \hat{\pi}_{y_2})$ be a KKT pair of \eqref{eq:kktf21-2}. By \cite[Theorem 2]{wachsmuth2013licq}, under full row rank of $B(\xi)$ in $Y_2(y_1, \xi)$,  the set of Lagrange multipliers $\{\hat{\pi}_{y_2}\}$ is a singleton. Then by \cite[Corollary 2.23]{shapiro2021lectures},  $\nabla_{y_1}  f_{21}(x_2, y_1, \xi)=\{-A(\xi)^\top \hat{\pi}_{y_2}\}$ is a singleton and $f_{21}(x_2, \cdot, \xi)$ is differentiable. Moreover, by the Danskin theorem, $\nabla_{x_2} f_{21}(x_2,  y_1, \xi)=\nabla_{x_2} F_2(x_2,  \hat{y}_2, \xi)$, where $\hat{y}_2$ is the unique solution of \eqref{eq:f21}.

Let $(\hat{{\bf y}}_2:\mathbb{R}^{n_2}\times Y_1\times \Xi\to\mathbb{R}^{m_2}, \hat{{\bm \pi}}_{y_2}:\mathbb{R}^{n_2}\times Y_1\times \Xi\to\mathbb{R}^{s_2})$ be the KKT pair mapping of \eqref{eq:kktf21-2}, such that $(\hat{{\bf y}}_2(x_2, y_1, \xi), \hat{{\bm \pi}}_{y_2}(x_2, y_1, \xi)) = (\hat{y}_2, \hat{\pi}_{y_2})$, the KKT pair of \eqref{eq:kktf21-2} with corresponding $(x_2, y_1, \xi)$.
Since $F_2(\cdot, \cdot, \xi)$ is  $\sigma$-strongly convex-strongly concave,  every element in $\partial_{y_2y_2}^2F_2(x_2, y_2, \xi)$ is a negative definite matrix. Then under Assumption \ref{a:second-stage} (i) and (iii), and Assumption \ref{a:xicompact-regular}, applying Theorem~\ref{t:imp} (in the Appendix) to the problem in \eqref{eq:f21} yields that $(\hat{{\bf y}}_2, \hat{{\bm \pi}}_{y_2})$ is Lipschitz continuous w.r.t. $(x_2, y_1)$, and continuous w.r.t. $\xi$, which implies $f_{21}(\cdot, \cdot, \xi)$ is  continuously differentiable for any given $\xi\in  \Xi$, $\nabla_{y_1} f_{21}(x_2, y_1, \xi)=-A(\xi)^\top \hat{{\bm \pi}}_{y_2}(x_2, y_1, \xi)$, $\nabla_{x_2} f_{21}(x_2, y_1, \xi)=\nabla_{x_2} F_2(x_2,  \hat{{\bf y}}_2(x_2, y_1, \xi), \xi)$,
$\nabla_{y_1} f_{21}$ and $\nabla_{x_2} f_{21}$ are Lipschitz continuous w.r.t. $(x_2, y_1)$, and continuous w.r.t. $\xi$.

The proof of part (ii) follows a similar argument as part (i) and will be omitted here.

 (iii) Note that
\begin{equation}\label{eq:bbfmin}
\psi_2(x_1, y_1, \xi):=\displaystyle{\min_{x_2\in X_2(x_1, \xi)}f_{21}(x_2,  y_1, \xi)}
\end{equation}
and $f_{21}(x_2,  \cdot, \xi)$ is  concave. For any $\lambda\in[0,1]$, and $y^1_1, y^2_1\in Y_1$, let
$$x_2^\lambda= \argmin_{x_2\in X_2(x_1, \xi)}f_{21}(x_2,  \lambda y^1_1 + (1-\lambda)y_1^2, \xi),$$
$x_2^1= \displaystyle{\argmin_{x_2\in X_2(x_1, \xi)}f_{21}(x_2,   y^1_1 , \xi)}$
and
$x_2^2= \displaystyle{\argmin_{x_2\in X_2(x_1, \xi)}f_{21}(x_2,   y^2_1 , \xi)}$. Then
for any given $x_1$ and $\xi$,  we have
$$
\begin{array}{lll}
\psi_2(x_1, \lambda y^1_1 + (1-\lambda)y_1^2, \xi) &=& f_{21}(x_2^\lambda,  \lambda y^1_1 + (1-\lambda)y_1^2, \xi)\\
&\geq& \lambda f_{21}(x_2^\lambda,   y^1_1, \xi) + (1-\lambda)f_{21}(x_2^\lambda,   y_1^2, \xi)\\
&\geq& \lambda f_{21}(x_2^1,   y^1_1, \xi) + (1-\lambda)f_{21}(x_2^2,   y_1^2, \xi)\\
&=&\lambda \psi_2(x_1,  y^1_1, \xi) + (1-\lambda)\psi_2(x_1, y_1^2, \xi),
\end{array}
$$
which implies the  concavity of $\psi_2$ w.r.t. $y_1$. Moreover, applying \cite[Proposition 2.21]{shapiro2021lectures} to \eqref{eq:bbfmin}, $\psi_2$ is  convex w.r.t. $x_1$.

Moreover, by continuous differentiability and strong convexity  of $f_{21}(\cdot,  y_1, \xi)$  from part (i),  every element in $\partial^2_{x_2x_2}f_{21}(x_2,  y_1, \xi)$ is a positive definite matrix. Then under Assumption \ref{a:second-stage} (i) and (iii), Assumption \ref{a:xicompact-regular}
 and applying Theorem~\ref{t:imp} (in the Appendix) to the problem in \eqref{eq:bbfmin}, similar as in part (i), $(\bar{{\bf x}}_2, {\bm \pi}_{x_2})$, the KKT pair of the problem in \eqref{eq:bbfmin}, is Lipschitz continuous w.r.t. $(x_1, y_1)$ and continuous w.r.t. $\xi$. Combining with \cite[Corollary 2.23]{shapiro2021lectures},  $\nabla_{x_1} \psi_2(x_1, y_1, \xi)=\{T(\xi)^\top {\bm \pi}_{x_2}(x_1, y_1, \xi)\}$ is a singleton, $\nabla_{x_1} \psi_2$ is Lipschitz continuous w.r.t. $(x_1, y_1)$, and continuous w.r.t. $\xi$. Obviously, $\psi_2(\cdot, y_1, \xi)$ is continuously differentiable.

Since $F_2$ is $\sigma$-strongly convex-strongly concave,
\begin{equation}\label{eq:bbfmin-y}
\psi_2(x_1, y_1, \xi)=\displaystyle{\max_{y_2\in Y_2(y_1, \xi)}f_{22}(x_1,  y_2, \xi)}.
\end{equation}
Then similar as above argument, $(\bar{{\bf y}}_2, {\bm \pi}_{y_2})$, the KKT pair of \eqref{eq:bbfmin-y}, is Lipschitz continuous w.r.t. $(x_1, y_1)$, and continuous w.r.t. $\xi$. Then $\nabla_{y_1} \psi_2(x_1, y_1, \xi)=\{-A(\xi)^\top {\bm \pi}_{y_2}(x_1, y_1, \xi)\}$ is a singleton, $\nabla_{y_1} \psi_2$ is Lipschitz continuous w.r.t. $(x_1, y_1)$  and continuous w.r.t. $\xi$,  and $\psi_2(x_1, \cdot, \xi)$ is continuously differentiable.

Note also that
$(\bar{{\bf x}}_2(x_1, y_1, \xi), \bar{{\bf y}}_2(x_1, y_1, \xi), {\bm \pi}_{x_2}(x_1, y_1, \xi), {\bm \pi}_{y_2}(x_1, y_1, \xi))$
is the unique solution of 
the KKT condition of the minimax problem in \eqref{eq:st-minimax}, then $(\bar{{\bf x}}_2(x_1, y_1, \xi), \bar{{\bf y}}_2(x_1, y_1, \xi))$ is the unique saddle point of the minimax problem in  \eqref{eq:st-minimax}.

(iv) Since $X_1\times Y_1\times \Xi$ is  compact   and $(\bar{{\bf x}}_2, \bar{{\bf y}}_2)$ is  continuous, for all $(x_1, y_1, \xi)\in X_1\times Y_1\times \Xi$, there exists convex and compact set $\hat{X}_2\times \hat{Y}_2\subset \mathbb{R}^{n_2}\times \mathbb{R}^{m_2}$ such that $(\bar{{\bf x}}_2(x_1, y_1, \xi), \bar{{\bf y}}_2(x_1, y_1, \xi))\in \hat{X}_2\times \hat{Y}_2$.

 By part (i), the solution function $\hat{{\bf y}}_2$ of the problem in \eqref{eq:f21} is continuous over $\hat{X}_2\times Y_1\times \Xi$. By the boundedness of $\hat{X}_2\times Y_1\times \Xi$, there exists convex and compact set $\check{Y}_2$ such that $\hat{Y}_2\subset\check{Y}_2$, $ \hat{{\bf y}}_2(x_2, y_1, \xi)\in \check{Y}_2$ over $\hat{X}_2\times Y_1\times \Xi$  and \eqref{eq:f21check} holds.

Similarly, we can prove the existence of $\check{X}_2$ such that $\hat{X}_2\subset\check{X}_2$  and \eqref{eq:f22check} holds.
 \hfill $\square$
\end{proof}

To study the  two-stage  stochastic minimax problem \eqref{eq:ts-minimax}-\eqref{eq:st-minimax}, we need the following definitions.

\begin{definition}\cite[Section 9.2.4]{shapiro2021lectures} It is said that functions $h_1:\mathbb{R}^n\times \Xi\to\mathbb{R}$ and $h_2:\mathbb{R}^n\times \Xi\to\mathbb{R}$  are random lsc and random upper semicontinuous (usc) respectively, if the epi-graphical multifunction $\xi\to \inmat{epi}~h_1(\cdot, \xi)$ and the hypo-graphical multifunction $\xi\to \inmat{hypo}~h_2(\cdot, \xi)$ are closed valued and measurable respectively.
\end{definition}

\begin{definition}
A function $h_3:\mathbb{R}^n\times \mathbb{R}^m\times \Xi\to\mathbb{R}$ is random lower-upper  semicontinuous w.r.t. $(x, y)$ if for any given $y$,
$h_3(\cdot, y, \cdot)$ is random lsc, and for any given $x$, $h_3(x, \cdot, \cdot)$ is random usc.
\end{definition}

\begin{proposition}\label{p:objv}
Under  Assumptions \ref{a:second-stage}-\ref{a:xicompact-regular}, the following statements hold.
\begin{itemize}

\item[(i)]
 $f_{21}$ is  random lower-upper semicontinuous on $\hat{X}_2\times Y_1$,  
$f_{22}$ is  random lower-upper semicontinuous  on $X_1\times \hat{Y}_2$,  where $\hat{X}_2\times \hat{Y}_2$ is a compact set containing all saddle points of problem \eqref{eq:st-minimax}. 

\item[(ii)] 
$\psi_2$ is random lower-upper semicontinuous  on $X_1\times Y_1$.

\item[(iii)]  
If $f(\cdot) + \psi_1(\cdot, y_1)$ is quasi-convex  for any $y_1\in Y_1$, $\lambda>0$ and $\tilde{f}:\mathbb{R}^{n_1}\to\mathbb{R}^{n_1}$ is a continuously differentiable vector-valued function, then there exists a local saddle point $(x_1^*, y_1^*)$ of  problem
\begin{equation}\label{eq:ts-minimax-0}
\min_{x_1\in X_1} \max_{y_1\in Y_1} \; \psi_\lambda(x_1, y_1): = F_1(x_1, y_1) +\lambda\|\tilde{f}(x_1)_+\|_0+ \bbe\left[\psi_2(x_1, y_1, \xi)\right].
\end{equation}
\end{itemize}
\end{proposition}

\begin{proof}
(i) Let $\tilde{Y}_2(\xi):=\{(y_1, y_2)\in Y_1\times \mathbb{R}^{m_2}: A(\xi)y_1+B(\xi)y_2\leq c(\xi)\}$.
From Assumption \ref{a:xicompact-regular}, if $\{\xi^k\}\subset \Xi$, $\xi^k\to\bar{\xi}\in \Xi$, $(y_1^k, y_2^k)\in \tilde{Y}_2(\xi^k)$ and $(y_1^k, y_2^k)\to(\bar{y}_1, \bar{y}_2)$, then $(\bar{y}_1, \bar{y}_2)\in \tilde{Y}_2(\bar{\xi})$, which implies $\tilde{Y}_2$ is closed. By \cite[Remark 62]{shapiro2021lectures}, $\tilde{Y}_2$ is a closed valued measurable multifunction.
Moreover, since $F_2$ is a Carath\'eodory function,  for given $x_2$, by \cite[Comments after Theorem 9.49]{shapiro2021lectures}, $F_2(x_2, \cdot, \cdot)+ \underline{\delta}_{\tilde{Y}_2(\xi)}(\cdot, \cdot)$ is a random usc function, where $\underline{\delta}_{\tilde{Y}_2(\xi)}(y_1, y_2)=0$ if $(y_1, y_2)\in \tilde{Y}_2(\xi)$ and $\underline{\delta}_{\tilde{Y}_2(\xi)}(y_1, y_2)=-\infty$ otherwise. Then by \cite[Theorem 9.49]{shapiro2021lectures}, $f_{21}(x_2,  \cdot, \cdot)$ is measurable.

Note that by Lemma \ref{l:2stage-bound-continuity} (iv), $\dom (F_2(x_2, \cdot, \xi) + \underline{\delta}_{Y_2(y_1, \xi)\cap\check{Y}_2}(\cdot))$ is nonempty and bounded for a.e. $\xi\in\Xi$,
according to \cite[Theorem 9.50]{shapiro2021lectures}, it follows that $ f_{21}(x_2,  \cdot, \cdot)$ is  random usc  for all $x_2\in\hat{X}_2$.

Then, we show that for any given $y_1\in Y_1$, $ f_{21}(\cdot,  y_1, \cdot)$ is random lsc. Since $F_2(\cdot, y_2, \cdot)$ is random lsc for given $y_2$. Let $\bar{x}_2\in \mathbb{R}^{n_2}$, $\{x_2^k\}\subset \mathbb{R}^{n_2}$ with $x_2^k\to\bar{x}_2$ as $k\to\infty$ and $y_2^1=\arg\max_{y_2\in Y_2(y_1, \xi)}F_2(\bar{x}_2, y_2, \xi)$, then we have
$$
\begin{array}{lll}
\displaystyle{\liminf_{k\to\infty} f_{21}(x_2^k, y_1, \xi)} &=&\displaystyle{\liminf_{k\to\infty}\max_{y_2\in Y_2(y_1, \xi)}F_{2}(x^k_2,  y_2, \xi)}\\
&\geq&\displaystyle{\liminf_{k\to\infty} F_{2}(x_2^k,  y^1_2, \xi)}\\
&\geq&F_{2}(\bar{x}_2,  y^1_2, \xi)\\
&=& f_{21}(\bar{x}_2,  y_1, \xi),
\end{array}
$$
which implies $ f_{21}$ is lsc w.r.t. $x_2$.

Applying \cite[Theorem 9.49]{shapiro2021lectures} to $\max_{y_2\in Y_2(y_1, \xi)}F_2(x_2, y_2, \xi)$, the optimal value function $f_{21}(\cdot, y_1, \cdot)$ is jointly measurable. Combining the measurability of $f_{21}(\cdot, y_1, \cdot)$  and the lower semicontinuity of $ f_{21}$ w.r.t. $x_2$, by \cite[Theorem 9.48]{shapiro2021lectures}, $ f_{21}(\cdot, y_1, \cdot)$ is random lsc for any given $y_1\in Y_1$. Then  $f_{21}$ is  random lower-upper semicontinuous on $\hat{X}_2\times Y_1$.

The proof for $f_{22}$ follows a similar argument as for $f_{21}$ and will be omitted here.

 (ii) 
Note that $ f_{21}(x_2,  \cdot, \cdot)$ is  random usc.
Let $\bar{y}_1\in Y_1$, $\{y_1^k\}\subset Y_1$ with $y_1^k\to \bar{y}_1$ as $k\to\infty$ and $x_2^1\in \argmin_{x_2\in X_2(x_1, \xi)} f_{21}(x_2,  \bar{y}_1, \xi)$. Then
$$
\begin{array}{lll}
\displaystyle{\limsup_{k\to\infty}\psi_2(x_1, y^k_1, \xi)} &=&\displaystyle{\limsup_{k\to\infty}\min_{x_2\in X_2(x_1, \xi)} f_{21}(x_2,  y^k_1, \xi)}\\
&\leq&\displaystyle{\limsup_{k\to\infty} f_{21}(x_2^1,  y^k_1, \xi)}\\
&\leq& f_{21}(x_2^1,  \bar{y}_1, \xi)\\
&=&\psi_2(x_1,\bar{y}_1, \xi),
\end{array}
$$
which implies that $\psi_2$ is usc w.r.t. $y_1$ for given $x_1$ and $\xi$.

Applying \cite[Theorem 9.49]{shapiro2021lectures} to problem \eqref{eq:bbfmin}, the minimax value function $\psi_2(x_1, \cdot, \cdot)$ is jointly measurable. Combining the measurability of $\psi_2(x_1, \cdot, \cdot)$ and the upper semicontinuity of $\psi_2$ w.r.t. $y_1$, by \cite[Theorem 9.48]{shapiro2021lectures}, $\psi_2(x_1, \cdot, \cdot)$ is random usc for any given $x_1\in X_1$.

Similar to the random upper semicontinuity of $\psi_2$ w.r.t. $y_1$ for any given $x_1\in X_1$, we can prove the random lower semicontinuity of $\psi_2$ w.r.t. $x_1$ for any given $y_1\in Y_1$. Then $\psi_2$ is random lower-upper semicontinuous w.r.t. $(x_1, y_1)$ on $X_1\times Y_1$. We omit the details here.


 Finally,  part (iii) is a corollary of \cite[Proposition 2.2]{bian2024nonsmooth}.
 \hfill $\square$
\end{proof}

\begin{remark}\label{r:strong-bounded}
Note that the second-stage minimax problem in \eqref{eq:st-minimax}  is a parametric minimax problem. In Lemma \ref{l:2stage-bound-continuity}, we investigate the continuity and the boundedness of the second-stage saddle point mapping $(\bar{{\bf x}}_2, \bar{{\bf y}}_2)$. In Proposition~\ref{p:objv}, we consider the random lower-upper semicontinuity of the minimax value function $\psi_2$ of the second-stage minimax problem in  \eqref{eq:st-minimax}.

Let
\begin{equation}\label{eq:ts-minimax-check}
\min_{x_1\in X_1} \max_{y_1\in Y_1} \; F_1(x_1, y_1) + \bbe\left[\hat{\psi}_2(x_1, y_1, \xi)\right]
\end{equation}
and
\begin{equation}\label{eq:st-minimax-check}
\hat{\psi}_2(x_1, y_1, \xi):=\min_{x_2\in X_2(x_1,\xi)\cap\check{X}_2}\max_{y_2\in Y_2(y_1, \xi)\cap\check{Y}_2}F_2(x_2, y_2, \xi),
\end{equation}
where $\check{X}_2$ and $\check{Y}_2$ are given in \eqref{eq:f21check}-\eqref{eq:f22check}.
By Lemma \ref{l:2stage-bound-continuity} and Proposition~\ref{p:objv}, it is easy to show the equivalence between \eqref{eq:ts-minimax}-\eqref{eq:st-minimax} and \eqref{eq:ts-minimax-check}-\eqref{eq:st-minimax-check} in the sense that: (i) for every $(x_1, y_1)\in X_1\times Y_1$ they have the same and unique saddle point $(\bar{{\bf x}}_2(x_1, y_1, \xi), \bar{{\bf y}}_2(x_1, y_1, \xi))$, and \eqref{eq:f21check}-\eqref{eq:f22check} hold; (ii)
they have same sets of points $(x_1^*, y_1^*)$ satisfying \eqref{eq:saddle21}.
\end{remark}

\subsection{Interchange of the
expectation and the minimax operator}\label{sec:31}

Two-stage stochastic minimization problem has been considered as an infinite-dimensional large-scale optimization problem \cite[Section 2.3.1]{shapiro2021lectures}. Can the two-stage stochastic minimax problem \eqref{eq:ts-minimax}-\eqref{eq:st-minimax}  also be treated as an infinite-dimensional large-scale minimax problem \eqref{eq:ts-minimax-interchange}? To answer this question, we  examine whether the expectation and the minimax operator in the second-stage problem can be interchanged.

\begin{proposition}\label{p:xy1212}
Under Assumptions \ref{a:second-stage}-\ref{a:xicompact-regular}, the two-stage stochastic minimax problem \eqref{eq:ts-minimax}-\eqref{eq:st-minimax} is equivalent to the minimax problem \eqref{eq:ts-minimax-interchange}.  
Moreover, the following statements hold.
\begin{itemize}

\item[(i)] If $(x_1^*, y_1^*)$ is a (local) saddle point of \eqref{eq:ts-minimax}, then $(x_1^*, {\bf x}_2^*, y_1^*, {\bf y}_2^*)$ is a (local) saddle point of \eqref{eq:ts-minimax-interchange}. If $(x_1^*, {\bf x}_2^*, y_1^*, {\bf y}_2^*)$ is a (local)saddle point of \eqref{eq:ts-minimax-interchange}, then $(x_1^*, y_1^*)$ is a (local) saddle point of \eqref{eq:ts-minimax}.

\item[(ii)]   If $f(\cdot) + \psi_1(\cdot, y_1)$ is quasi-convex over $X_1$ for any given $y_1\in Y_1$, then there exists a saddle point $(x_1^*, {\bf x}_2^*,y_1^*, {\bf y}_2^*)$ of \eqref{eq:ts-minimax-interchange}.

\item[(iii)]  In addition to (ii),
if 
$\lambda>0$
and $\tilde{f}$ is a continuously differentiable vector-valued function, then there exists a  local saddle point $(x_1^*, {\bf x}_2^*,y_1^*, {\bf y}_2^*)$ of
\begin{equation*}\label{eq:ts-minimax-interchange-hat}
\min_{(x_1,  {\bf x}_2)\in {\bf X}} \max_{(y_1, {\bf y}_2)\in {\bf Y}} \; F_1(x_1, y_1) + \lambda\|\tilde{f}(x_1)_+\|_0 + \bbe\left[F_2({\bf x}_2(\xi), {\bf y}_2(\xi), \xi)\right].
\end{equation*}
\end{itemize}

\end{proposition}

\begin{proof}
For any given $(x_1, y_1)\in X_1\times Y_1$, we consider the expectation of the second-stage problem
\begin{equation}\label{eq:E-minimax1}
\bbe\left[\min_{ x_2\in X_2(x_1, \xi)}\max_{y_2\in Y_2(y_1, \xi)}F_2(x_2, y_2, \xi)\right]
\end{equation}
firstly. Note that by Remark \ref{r:strong-bounded}, \eqref{eq:E-minimax1} is equivalent to
\begin{equation}\label{eq:E-minimax}
\bbe\left[\min_{ x_2\in X_2(x_1, \xi)\cap \check{X}_2}\max_{y_2\in Y_2(y_1, \xi)\cap\check{Y}_2}F_2(x_2, y_2, \xi)\right].
\end{equation}
By Proposition \ref{p:objv} (i),   $\displaystyle{\max_{ y_2\in Y_2(y_1, \xi)\cap\check{Y}}F_2(\cdot,  y_2, \xi)}$ is random lsc over $\hat{X}_2$. Then by  \cite[Theorem 9.108]{shapiro2021lectures}, \eqref{eq:E-minimax} is equivalent to
\begin{equation}\label{eq:min-Emax}
\min_{{\bf x_2}\in \tilde{X}_2(x_1)}\bbe\left[\max_{y_2\in Y_2(y_1, \xi)\cap\check{Y}_2}F_2({\bf x}_2(\xi),  y_2, \xi)\right],
\end{equation}
where $\tilde{X}_2(x_1):=\left\{{\bf x}_2\in\mathcal{L}^{n_2}_p: {\bf x}_2(\xi)\in\hat{X}_2,  T(\xi)x_1+W(\xi){\bf x}_2(\xi)\leq h(\xi), \;a.e.\; \xi\in\Xi\right\}$.
Moreover, by Lemma \ref{l:2stage-bound-continuity} (iii), the optimal solution function $\bar{{\bf x}}_2(x_1, y_1, \cdot)$ of problem \eqref{eq:min-Emax} is continuous. Then, since $F_2(x_2,  \cdot, \xi)$ is random usc, we have $F_2(\bar{{\bf x}}_2(x_1, y_1, \xi),  \cdot, \xi)$ is random usc, and  by   \cite[Theorem 9.108]{shapiro2021lectures}, \eqref{eq:min-Emax} is equivalent to
\begin{equation}\label{eq:min-max-E}
\min_{{\bf x_2}\in \tilde{X}_2(x_1)}\max_{{\bf y}_2(\xi)\in \tilde{Y}_2(y_1)}\bbe\left[F_2({\bf x}_2(\xi), {\bf y}_2(\xi), \xi)\right],
\end{equation}
where $\tilde{Y}_2(y_1):=\left\{{\bf y}_2\in\mathcal{L}^{m_2}_p:  {\bf y}_2(\xi)\in\hat{Y}_2, A(\xi)y_1+B(\xi){\bf y}_2(\xi)\leq c(\xi), \;a.e.\; \xi\in\Xi\right\}$.
Then for any $(x_1, y_1)\in X_1\times Y_1$,
$$
\bbe\left[\min_{x_2\in X_2(x_1, \xi)}\max_{y_2\in Y_2(y_1, \xi)}F_2(x_2, y_2, \xi)\right]=\min_{{\bf x_2}\in \tilde{X}_2(x_1)}\max_{{\bf y}_2(\xi)\in \tilde{Y}_2(y_1)}\bbe\left[F_2({\bf x}_2(\xi), {\bf y}_2(\xi), \xi)\right].
$$
Given that Lemma \ref{l:2stage-bound-continuity}(iv) ensures
 $(\bar{{\bf x}}_2(x_1, y_1, \xi), \bar{{\bf y}}_2(x_1, y_1, \xi))\in\hat{X}_2\times \hat{Y}_2\subset\check{X}_2\times \check{Y}_2$ for all $(x_1, y_1, \xi)\in X_1\times Y_1\times \Xi$, it follows that the two-stage stochastic minimax problem
\eqref{eq:ts-minimax}-\eqref{eq:st-minimax} is equivalent to \eqref{eq:ts-minimax-interchange}.

Then we consider (i).  We only show the forward implication, since the backward implication follows the forward implication immediately.
Since $(x_1^*, y_1^*)$ is a saddle point of \eqref{eq:ts-minimax},
\begin{equation}\label{eq:saddlepoint1-eq}
\begin{array}{lll}
&&\displaystyle{\min_{(x_1,  {\bf x}_2)\in {\bf X}} \max_{(y_1, {\bf y}_2)\in {\bf Y}}} F_1(x_1, y_1) + \bbe\left[F_2({\bf x}_2(\xi), {\bf y}_2(\xi), \xi)\right]\\
 &=& \displaystyle{\min_{x_1\in X_1} \max_{y_1\in Y_1}} F_1(x_1, y_1) + \bbe\left[\psi_2(x_1, y_1, \xi)\right]\\
  &=&F_1(x^*_1, y^*_1) + \bbe\left[\psi_2(x^*_1, y^*_1, \xi)\right]\\
&=& \displaystyle{\max_{y_1\in Y_1} \min_{x_1\in X_1}} F_1(x_1, y_1) + \bbe\left[\psi_2(x_1, y_1, \xi)\right] \\
&=& \displaystyle{\max_{(y_1, {\bf y}_2)\in {\bf Y}} \min_{(x_1,  {\bf x}_2)\in {\bf X}}} F_1(x_1, y_1) + \bbe\left[F_2({\bf x}_2(\xi), {\bf y}_2(\xi), \xi)\right],
\end{array}
\end{equation}
where the second and third equalities are from \cite[Theorem 1.4.1]{facchinei2003finite}, and the last equality follows from a similar argument as the equivalence between \eqref{eq:ts-minimax}-\eqref{eq:st-minimax}  and \eqref{eq:ts-minimax-interchange}.

By Lemma \ref{l:2stage-bound-continuity} (iv)
\[
\begin{array}{lll}
z_1(x_1,{\bf x}_2):&=&\max_{y_1\in Y_1} F_1(x_1, y_1) + \bbe[ f_{21}({\bf x}_2(\xi), y_1, \xi)]\\
&=& \max_{(y_1, {\bf y}_2)\in \bar{{\bf Y}}} \; F_1(x_1, y_1) + \bbe\left[F_2({\bf x}_2(\xi), {\bf y}_2(\xi), \xi)\right],
\end{array}
\]
where
\begin{multline*}
\bar{{\bf Y}}:=\{(y_1, {\bf y}_2)\in\mathbb{R}^{m_1}\times\mathcal{L}^{m_2}_p: y_1\in Y_1, {\bf y}_2(\xi)\in \check{Y}_2, \\
A(\xi)y_1+B(\xi){\bf y}_2(\xi)\leq c(\xi), \;a.e.\; \xi\in\Xi\}.
\end{multline*}
Note that $Y_1$ and  $\check{Y}_2$ are compact, $\bar{{\bf Y}}$ is weakly compact. Then by \cite[Proposition 4.4]{bonnans2013perturbation}, $z_1$ is continuous w.r.t. $(x_1,{\bf x}_2)$ over ${\bf X}$. Moreover, since $X_1$ is compact and by Lemma \ref{l:2stage-bound-continuity} (i) that $f_{21}(\cdot, \cdot, \xi)$ is strongly convex-concave for any given $\xi\in\Xi$,
there exists $(x_1^*, {\bf x}_2^*)\in {\bf X}$ such that
\begin{equation}\label{eq:saddlepoint2}
\min_{(x_1,  {\bf x}_2)\in {\bf X}} z_1(x_1,{\bf x}_2)=z_1(x_1^*, {\bf x}_2^*)\\
= \max_{(y_1, {\bf y}_2)\in {\bf Y}}F_1(x_1^*, y_1) + \bbe\left[F_2({\bf x}_2^*(\xi), {\bf y}_2(\xi), \xi)\right].
\end{equation}
  Similarly,  there exists $(y_1^*, {\bf y}_2^*)\in {\bf Y}$ such that
 \begin{equation}\label{eq:saddlepoint3}
 \max_{(y_1,  {\bf y}_2)\in {\bf Y}} z_2(y_1,{\bf y}_2)=z_2(y_1^*, {\bf y}_2^*)= \min_{(x_1,  {\bf x}_2)\in {\bf X}}F_1(x_1, y_1^*) + \bbe\left[F_2({\bf x}_2(\xi), {\bf y}^*_2(\xi), \xi)\right],
 \end{equation}
 where
 $z_2(y_1, {\bf y}_2):=\min_{x_1\in X_1} F_1(x_1, y_1) + \bbe[ f_{22}(x_1, {\bf y}_2(\xi), \xi)]$ is continuous w.r.t. $(y_1, {\bf y}_2)$.
Combining \eqref{eq:saddlepoint2}-\eqref{eq:saddlepoint3}, we have
\[
\begin{array}{lll}
&&\displaystyle{\max_{(y_1, {\bf y}_2)\in {\bf Y}}}F_1(x_1^*, y_1) + \bbe\left[F_2({\bf x}_2^*(\xi), {\bf y}_2(\xi), \xi)\right]\\
&=&F_1(x_1^*, y_1^*) + \bbe\left[F_2({\bf x}_2^*(\xi), {\bf y}_2^*(\xi), \xi)\right]\\
&=&\displaystyle{\min_{(x_1,  {\bf x}_2)\in {\bf X}}}F_1(x_1, y_1^*) + \bbe\left[F_2({\bf x}_2(\xi), {\bf y}_2^*(\xi), \xi)\right],
\end{array}
\]
which implies that $(x_1^*, {\bf x}_2^*,y_1^*, {\bf y}_2^*)$ is a saddle point.

For a local saddle point $(x_1^*, y_1^*)$  of \eqref{eq:ts-minimax}, the strong convexity-strong concavity of $F_2$  w.r.t. $(x_2, y_2)$ ensures that the local saddle point $(\bar{{\bf x}}(x_1^*, y_1^*, \xi), \bar{{\bf y}}(x_1^*, y_1^*, \xi))$ of the second-stage minimax problem in \eqref{eq:st-minimax} is also the unique saddle point. Then let $({\bf x}_2^*, {\bf y}_2^*):=(\bar{{\bf x}}(x_1^*, y_1^*, \cdot), \bar{{\bf y}}(x_1^*, y_1^*, \cdot))$, and following the above argument, $(x_1^*, {\bf x}_2^*,y_1^*, {\bf y}_2^*)$ is a local saddle point of \eqref{eq:ts-minimax-interchange}, and vice versa.

With part (i), we consider part (ii),
the existence of a saddle point. Note that since $f(\cdot) + \psi_1(\cdot, y_1)$ is quasi-convex over $X_1$ for any given $y_1\in Y_1$, $F_1$ is quasi-convex w.r.t. $x_1$, and then
\begin{equation*}\label{eq:saddlepoint1}
\begin{array}{lll}
&&\displaystyle{\min_{(x_1,  {\bf x}_2)\in {\bf X}} \max_{(y_1, {\bf y}_2)\in {\bf Y}}} F_1(x_1, y_1) + \bbe\left[F_2({\bf x}_2(\xi), {\bf y}_2(\xi), \xi)\right]\\
 &=& \displaystyle{\min_{x_1\in X_1} \max_{y_1\in Y_1}} F_1(x_1, y_1) + \bbe\left[\psi_2(x_1, y_1, \xi)\right]\\
&=& \displaystyle{\max_{y_1\in Y_1} \min_{x_1\in X_1}} F_1(x_1, y_1) + \bbe\left[\psi_2(x_1, y_1, \xi)\right] \\
&=& \displaystyle{\max_{(y_1, {\bf y}_2)\in {\bf Y}} \min_{(x_1,  {\bf x}_2)\in {\bf X}}} F_1(x_1, y_1) + \bbe\left[F_2({\bf x}_2(\xi), {\bf y}_2(\xi), \xi)\right],
\end{array}
\end{equation*}
where the second equality follows from Sion \cite{Sion58}, and the last equality follows from the similar argument as the equivalence between \eqref{eq:ts-minimax}-\eqref{eq:st-minimax}  and \eqref{eq:ts-minimax-interchange}.  Moreover, similar as the  argument after \eqref{eq:saddlepoint1-eq} in part (i),  $(x_1^*, {\bf x}_2^*,y_1^*, {\bf y}_2^*)$ is a saddle point of \eqref{eq:ts-minimax-interchange}.

Finally, part (iii) follows from part (ii) and  \cite[Proposition 2.2]{bian2024nonsmooth} directly.
 \hfill $\square$
\end{proof}

Note that under condition (ii)  of Proposition \ref{p:xy1212},  problem \eqref{eq:ts-minimax}-\eqref{eq:st-minimax} is equivalent to  two-stage stochastic two-player zero-sum game \eqref{eq:tzs-game1}-\eqref{eq:tzs-game2}.

\begin{remark}
We consider three kinds of saddle points for two-stage stochastic minimax problem \eqref{eq:ts-minimax}-\eqref{eq:st-minimax}: the saddle point $(x_1^*, y_1^*)$  of \eqref{eq:ts-minimax}, the saddle point mapping  $(\bar{{\bf x}}_2(x_1, y_1, \xi), \bar{{\bf y}}_2(x_1, y_1, \xi))$  of the second-stage minimax problem in  \eqref{eq:st-minimax} (see Definition \ref{d:saddle-mapping-2}), and the saddle point
$(x_1^*, {\bf x}_2^*,y_1^*, {\bf y}_2^*)$  of \eqref{eq:ts-minimax-interchange}.

Under Assumptions \ref{a:second-stage}-\ref{a:xicompact-regular},  by Proposition \ref{p:xy1212} (i), $(x_1^*, y_1^*)$ is a saddle point of \eqref{eq:ts-minimax}  if and only if $(x_1^*, {\bf x}_2^*,y_1^*, {\bf y}_2^*)$ is a saddle point of \eqref{eq:ts-minimax-interchange}, and by Lemma \ref{l:2stage-bound-continuity},
\[
(\bar{{\bf x}}_2(x^*_1, y^*_1, \cdot), \bar{{\bf y}}_2(x^*_1, y^*_1, \cdot)) = ( {\bf x}_2^*(\cdot),  {\bf y}_2^*(\cdot))
\]
 and $(\bar{{\bf x}}_2(\cdot, \cdot, \cdot), \bar{{\bf y}}_2(\cdot, \cdot, \cdot))$ is continuous.
\end{remark}

\begin{remark}
Note that we can not apply Proposition \ref{p:kkt} to the two-stage stochastic minimax problem with cardinality penalties \eqref{eq:ts-minimax-0} since $\lambda\|\tilde{f}(\cdot)_+\|_0$ is not lsc. To overcome this difficulty, we may consider the continuous relaxation $r(\tilde{f}(\cdot), \mu)$ of $\lambda\|\tilde{f}(\cdot)_+\|_0$, where $r(\cdot, \cdot)$ is defined in \cite[Formulation (3.3)]{bian2024nonsmooth}. Then we can apply Proposition \ref{p:kkt} to the continuous relaxation problem of \eqref{eq:ts-minimax-0} as follows:
\begin{equation}\label{eq:ts-minimax-tilde}
\min_{x_1\in X_1} \max_{y_1\in Y_1} \; \tilde{\psi}_\lambda(x_1, y_1): = F_1(x_1, y_1) +\lambda r(\tilde{f}(x_1), \mu)+ \bbe\left[\psi_2(x_1, y_1, \xi)\right].
\end{equation}
By \cite[Theorem 4.1]{bian2024nonsmooth}, under suitable conditions (e.g. \cite[Assumption 4.1]{bian2024nonsmooth}),
$(x_1^*, y_1^*)$ is a (local) saddle point of \eqref{eq:ts-minimax-0} if and only if it is a (local) saddle point of \eqref{eq:ts-minimax-tilde}.

\end{remark}

\subsection{Global and local minimax points}

Since  \eqref{eq:ts-minimax} is a nonconvex-nonsmooth minimax problem, the sets of  saddle points and local saddle points of \eqref{eq:ts-minimax} may be empty. We therefore consider the sets of global minimax points and local minimax points of \eqref{eq:ts-minimax}, denoted by $S_g$ and $S_l$, respectively.

\begin{proposition}\label{p:kkt}
Suppose that Assumptions \ref{a:second-stage}-\ref{a:xicompact-regular} hold. Then (a) $x_1^*$ is a stationary point of problem \eqref{eq:first-stage-min} if and only if  (b) there exists a corresponding $y_1^*$ such that $(x_1^*, y_1^*)$ is a solution of the following variational inequality: 
\begin{equation}\label{eq:kkt}
\left\{
\begin{aligned}
&0\in \partial f(x_1) + \nabla_{x_1}\psi_1(x_1, y_1) + \nabla_{x_1}\bbe\left[\psi_2(x_1, y_1, \xi)\right]+ \mathcal{N}_{X_1}(x_1),\\
&0\in -\nabla_{y_1}\psi_1(x_1, y_1) - \nabla_{y_1}\bbe\left[\psi_2(x_1, y_1, \xi)\right]+ \partial g(y_1)+ \mathcal{N}_{Y_1}(y_1).
\end{aligned}
\right.
\end{equation}
Moreover, (c) $S_g\subset S_l\subset S_{kkt}$ are nonempty, where $S_{kkt}$ is  the solution set of \eqref{eq:kkt}.
\end{proposition}
\begin{proof}
We prove (a) $\Rightarrow$ (b) firstly. When $x_1^*$ is a stationary point of problem \eqref{eq:first-stage-min}, by Lemma~\ref{l:2stage-bound-continuity} (iii) and Proposition \ref{p:objv} (ii), $\bbe[\psi_2]$ is convex-concave and continuous, combined with  strong convexity of $g(\cdot)- \psi_1(x_1, \cdot)$,  there must exist a unique $y_1^*$ such that
\[
0\in -\nabla_{y_1}\psi_1(x_1^*, y_1^*) - \nabla_{y_1}\bbe\left[\psi_2(x_1^*, y_1^*, \xi)\right]+ \partial g(y_1^*)+ \mathcal{N}_{Y_1}(y_1^*),
\]
and
$\bar{\psi}(x_1^*)=\psi(x_1^*, y_1^*)$. Moreover, by Danskin's theorem,
\begin{equation}\label{eq:partial-barpsi}
\partial\bar{\psi}(x_1^*)=\partial_{x_1}\psi(x_1^*, y_1^*)=\partial f(x_1^*) + \nabla_{x_1}\psi_1(x_1^*, y_1^*) + \nabla_{x_1}\bbe\left[\psi_2(x_1^*, y_1^*, \xi)\right].
\end{equation}
Then $0\in \partial\bar{\psi}(x_1^*) + \mathcal{N}_{X_1}(x_1^*)$ implies that $(x_1^*, y_1^*)$ is a solution of \eqref{eq:kkt}.

The proof of (b) $\Rightarrow$ (a) is similar and we omit the details.

Now we prove (c). By \eqref{eq:partial-barpsi}, $\bar{\psi}(\cdot)$ is Lipschitz continuous over $X_1$. Then, by the compactness of $X_1$, the global optimal solution set of \eqref{eq:first-stage-min}, denoted by $S^*_{\bar{\psi}}$, is nonempty. Then it is easy to check that
\[
S_g:=\left\{(\bar{x}_1^*, \bar{\bm y}_1^*(\bar{x}_1^*)): \bar{x}_1^*\in S^*_{\bar{\psi}} \inmat{ and } \bar{\bm y}_1^*(\bar{x}_1^*):=\arg\max_{y_1\in Y_1} \psi(\bar{x}_1^*, y_1)\right\}.
\]

By Remark \ref{r:localminimax}, if $\bar{y}_1$ is a global maximum point of $\psi(\bar{x}_1, \cdot)$, and $\bar{x}_1$ is a local minimum point of $\bar{\psi}(\cdot)$, then $(\bar{x}_1, \bar{y}_1)$ is a local minimax point of problem \eqref{eq:ts-minimax}. It is trivial to observe that all global minimax points satisfy the above conditions and are therefore local minimax points, that is,  $S_g\subset S_l$.


Moreover,  $S_l\subset S_{kkt}$ follows directly from \cite[Theorem 3.11]{jiang2023optimality}.
 \hfill $\square$
\end{proof}

Note that  problem \eqref{eq:ts-minimax-interchange} is also a nonconvex-nonsmooth minimax problem; consequently, the set of  saddle points may be empty. In this case, we consider  global minimax points and local minimax points of \eqref{eq:ts-minimax-interchange}.

\begin{proposition}\label{p:gg-ll}
Under Assumptions \ref{a:second-stage}-\ref{a:xicompact-regular}, the following statements 
hold.
\begin{itemize}
\item[(i)] $(\tilde{x}_1, \tilde{y}_1)$ is a global minimax point of \eqref{eq:ts-minimax}  if and only if $(\tilde{x}_1, \tilde{{\bf x}}_2, \tilde{y}_1, \tilde{{\bf y}}_2)$ is a global minimax point of \eqref{eq:ts-minimax-interchange}, where $(\tilde{{\bf x}}_2(\xi), \tilde{{\bf y}}_2(\xi)) = ({\bf x}^*_2(\tilde{x}_1, \tilde{y}_1, \xi), {\bf y}^*_2(\tilde{x}_1, \tilde{y}_1, \xi))$ is the unique saddle point of the minimax problem in \eqref{eq:st-minimax} with $(\tilde{x}_1, \tilde{y}_1)$ for a.e. $\xi\in\Xi$;

\item[(ii)] $(\tilde{x}_1, \tilde{y}_1)$ is a local minimax point of \eqref{eq:ts-minimax}  if and only if $(\tilde{x}_1, \tilde{{\bf x}}_2, \tilde{y}_1, \tilde{{\bf y}}_2)$ is a local minimax point of \eqref{eq:ts-minimax-interchange}, where $(\tilde{{\bf x}}_2(\xi), \tilde{{\bf y}}_2(\xi)) = ({\bf x}^*_2(\tilde{x}_1, \tilde{y}_1, \xi), {\bf y}^*_2(\tilde{x}_1, \tilde{y}_1, \xi))$ is the unique saddle point of the minimax problem in \eqref{eq:st-minimax} with $(\tilde{x}_1, \tilde{y}_1)$ for a.e. $\xi\in\Xi$.
\end{itemize}

\end{proposition}

\begin{proof}
 (i)  We only show the backward implication, since the forward implication follows from the strong convexity-concavity of the second-stage problem  immediately.
 For any global minimax point $(\tilde{x}_1, \tilde{{\bf x}}_2,\tilde{y}_1, \tilde{{\bf y}}_2)\in {\bf X}\times {\bf Y}$ of the minimax problem \eqref{eq:ts-minimax-interchange}, we have
\begin{align*}
F_1(\tilde{x}_1, \tilde{y}_1) + \mathbb{E}\left[F_2(\tilde{{\bf x}}_2(\xi), {\bf y}_2(\xi), \xi)\right] &\leq F_1(\tilde{x}_1, \tilde{y}_1) + \mathbb{E}\left[F_2(\tilde{{\bf x}}_2(\xi), \tilde{{\bf y}}_2(\xi), \xi)\right] \\
&\leq \max_{{\bf y}'_2(\xi)\in {\bf Y}_2(\tilde{y}_1)}F_1(\tilde{x}_1, \tilde{y}_1) + \mathbb{E}\left[F_2({\bf x}_2(\xi), {\bf y}'_2(\xi), \xi)\right]
\end{align*}
 for any $({\bf x}_2, {\bf y}_2)\in {\bf X}_2(\tilde{x}_1)\times {\bf Y}_2(\tilde{y}_1)$ with
\[
{\bf X}_2(\tilde{x}_1):=\{{\bf x}_2\in\mathcal{L}^{n_2}_p: {\bf x}_2(\xi)\in \check{X}_2, T(\xi)\tilde{x}_1+W(\xi) {\bf x}_2(\xi)\leq h(\xi), \;a.e.\; \xi\in\Xi\}
\]
and
\[
{\bf Y}_2(\tilde{y}_1):=\{{\bf y}_2\in\mathcal{L}^{m_2}_p:  {\bf y}_2(\xi)\in \check{Y}_2, A(\xi)\tilde{y}_1+B(\xi){\bf y}_2(\xi)\leq c(\xi), \;a.e.\; \xi\in\Xi\},
\]
which implies $(\tilde{{\bf x}}_2(\xi), \tilde{{\bf y}}_2(\xi))$ is the global minimax point of the second-stage minimax problem in \eqref{eq:st-minimax} with $(\tilde{x}_1, \tilde{y}_1)$ for a.e. $\xi\in\Xi$.  Since by Lemma \ref{l:2stage-bound-continuity}, the minimax problem in \eqref{eq:st-minimax} is strongly convex-strongly concave,  $(\tilde{{\bf x}}_2(\xi), \tilde{{\bf y}}_2(\xi)) = ({\bf x}^*_2(\tilde{x}_1, \tilde{y}_1, \xi), {\bf y}^*_2(\tilde{x}_1, \tilde{y}_1, \xi))$ is the unique saddle point of the second-stage minimax problem in \eqref{eq:st-minimax}, and then \eqref{eq:ts-global} is equivalent to \eqref{eq:fs-global}.

The proof of (ii) is similar as (i), we omit the details.
 \hfill $\square$
\end{proof}

\begin{remark}
If $F_1$ is continuously differentiable, then the KKT condition of the two-stage minimax problem \eqref{eq:ts-minimax}-\eqref{eq:st-minimax} is
\begin{equation}\label{eq:tsSVI}
\left\{
\begin{aligned}
&0\in \nabla_{x_1} F_1(x_1, y_1) + \mathbb{E}[T(\xi)^\top {\bm \pi}_x(\xi)] + \mathcal{N}_{X_1}(x_1), \\
&0\in -\nabla_{y_1} F_1(x_1, y_1) + \mathbb{E}[A(\xi)^\top {\bm \pi}_y(\xi)] + \mathcal{N}_{Y_1}(y_1), \\
&0= \nabla_{{\bf x}_2(\xi)}F_2({\bf x}_2(\xi), {\bf y}_2(\xi), \xi) + W(\xi)^\top {\bm \pi}_x(\xi), \\
&0\leq {\bm \pi}_x(\xi) \bot h(\xi)-T(\xi)x_1-W(\xi){\bf x}_2(\xi)\geq0, \\
&0= -\nabla_{{\bf y}_2(\xi)}F_2({\bf x}_2(\xi), {\bf y}_2(\xi), \xi) + B(\xi)^\top {\bm \pi}_y(\xi), \\
&0\leq {\bm \pi}_y(\xi) \bot c(\xi)-A(\xi)y_1-B(\xi){\bf y}_2(\xi)\geq0,  ~~~~\mathrm{a.e.}\;\xi\in\Xi,
\end{aligned}
\right.
\end{equation}
where ${\bf x}_2\in\mathcal{L}_p^{n_2}, {\bf y}_2\in\mathcal{L}_p^{m_2}$,  ${\bm \pi}_x\in \mathcal{L}_q^{l_2}$ and ${\bm \pi}_y\in \mathcal{L}_q^{s_2}$. 
Clearly, \eqref{eq:tsSVI} is a two-stage SVI. 

\end{remark}

\section{Sample average approximation of the two-stage stochastic minimax problem}

In this section, we consider SAA of  two-stage stochastic minimax problem \eqref{eq:ts-minimax}-\eqref{eq:st-minimax}. Let $\xi^1, \cdots, \xi^N$ be  independent and identically distributed (i.i.d.) samples of random variable $\xi$. We consider the following SAA problem:
\begin{equation}\label{eq:ts-minimax-SAA}
\min_{x_1\in X_1} \max_{y_1\in Y_1} \; \psi_N(x_1, y_1): = f(x_1) + \psi_1(x_1, y_1) + \Phi_{N}(x_1, y_1) - g(y_1),
\end{equation}
where $\Phi_{N}(x_1, y_1):=\frac{1}{N}\sum_{i=1}^N\psi_2(x_1, y_1, \xi^i) $,  $\psi_2(x_1, y_1, \xi^i)$ is defined in \eqref{eq:st-minimax}.

Let $\bar{\psi}_N(x_1):=\max_{y_1\in Y_1}\psi_N(x_1, y_1)$. Then $\bar{x}_1^N$ is a stationary point of
\begin{equation}\label{eq:first-stage-min-SAA}
\min_{x_1\in X_1} \; \bar{\psi}_N(x_1),
\end{equation}
if $0\in \partial \bar{\psi}_N(\bar{x}_1^N)+\mathcal{N}_{X_1}(\bar{x}_1^N)$.

\begin{proposition}\label{p:kkt-SAA}
Suppose that Assumptions \ref{a:second-stage}-\ref{a:xicompact-regular} hold. Then $\bar{x}_1^N$ is a stationary point of problem \eqref{eq:first-stage-min-SAA} if and only if  there exists a corresponding $\bar{y}_1^N$ such that $(\bar{x}_1^N, \bar{y}_1^N)$ is a solution of
\begin{equation}\label{eq:kkt-SAA}
\left\{
\begin{aligned}
&0\in \partial f(x_1) + \nabla_{x_1}\psi_1(x_1, y_1) + \nabla_{x_1}\Phi_{N}(x_1, y_1)+ \mathcal{N}_{X_1}(x_1),\\
&0\in -\nabla_{y_1}\psi_1(x_1, y_1) - \nabla_{y_1}\Phi_{N}(x_1, y_1)+ \partial g(y_1)+ \mathcal{N}_{Y_1}(y_1).
\end{aligned}
\right.
\end{equation}
\end{proposition}
Note that $\nabla_{x_1}\Phi_{N}(x_1, y_1)=\frac{1}{N}\sum_{i=1}^N\nabla_{x_1}\psi_2(x_1, y_1, \xi^i)$ and $\nabla_{y_1}\Phi_{N}(x_1, y_1)=\frac{1}{N}\sum_{i=1}^N\nabla_{y_1}\psi_2(x_1, y_1, \xi^i)$.
The proof of Proposition \ref{p:kkt-SAA} is similar as the proof of Proposition \ref{p:kkt}, here we omit the details.

Then we consider the convergence between \eqref{eq:ts-minimax} and \eqref{eq:ts-minimax-SAA} in the sense of  stationary points of first-stage minimization problem \eqref{eq:first-stage-min} and its SAA problem \eqref{eq:first-stage-min-SAA}. By Propositions \ref{p:kkt} and \ref{p:kkt-SAA}, that is equivalent to  the convergence of solutions between \eqref{eq:kkt} and \eqref{eq:kkt-SAA}. We need the following notations.

For simplicity of notation, we set
\begin{equation}\label{eq:G1}
\begin{pmatrix}
G_x(x_1, y_1)\\
G_y(x_1, y_1)
\end{pmatrix}
:=
\begin{pmatrix}
\partial f(x_1) + \nabla_{x_1}\psi_1(x_1, y_1) + \nabla_{x_1}\bbe\left[\psi_2(x_1, y_1, \xi)\right]\\
-\nabla_{y_1}\psi_1(x_1, y_1) - \nabla_{y_1}\bbe\left[\psi_2(x_1, y_1, \xi)\right]+ \partial g(y_1)
\end{pmatrix}
\end{equation}
and
\begin{equation}\label{eq:G1N}
\begin{pmatrix}
G_x^N(x_1, y_1)\\
G_y^N(x_1, y_1)
\end{pmatrix}
:=
\begin{pmatrix}
\partial f(x_1) + \nabla_{x_1}\psi_1(x_1, y_1) + \nabla_{x_1}\Phi_{N}(x_1, y_1)\\
-\nabla_{y_1}\psi_1(x_1, y_1) - \nabla_{y_1}\Phi_{N}(x_1, y_1)+ \partial g(y_1)
\end{pmatrix}.
\end{equation}

%
%
%

\begin{theorem}
Suppose that Assumptions \ref{a:second-stage}-\ref{a:xicompact-regular} hold and $\{\xi^1, \cdots, \xi^N\}$ is a set of  i.i.d. samples of random variable $\xi$. Let $S^*$ be the solution set of SVI \eqref{eq:kkt} and $\{(\bar{x}_1^N, \bar{y}_1^N)\}$ be a sequence of solutions of its SAA problem \eqref{eq:kkt-SAA}. Then almost surely (a.s.)
\begin{equation}\label{eq:asconvergence}
 \lim_{N\to\infty}\mathbb{D}((\bar{x}_1^N, \bar{y}_1^N), S^*)=0.
 \end{equation}
%
\end{theorem}

\begin{proof}
 Note that, by using notation \eqref{eq:G1} and \eqref{eq:G1N}, SVI \eqref{eq:kkt} and its SAA problem \eqref{eq:kkt-SAA} can be written as
\[
0\in \begin{pmatrix}
G_x(x_1, y_1)\\
G_y(x_1, y_1)
\end{pmatrix}
+ \mathcal{N}_{X_1\times Y_1}((x_1, y_1)), \; {\rm and}\;
0\in \begin{pmatrix}
G_x^N(x_1, y_1)\\
G_y^N(x_1, y_1)
\end{pmatrix}
+ \mathcal{N}_{X_1\times Y_1}((x_1, y_1)).
\]
Moreover, since $\partial f(\cdot)$ and $\partial g(\cdot)$ are outer semicontinuous (osc)\footnote{A multifunction $S: \mathbb{R}^n\rightrightarrows\mathbb{R}^n$ is osc, if for all $\epsilon>0$, there exists $\delta>0$ such that $S(x')\subset S(x) +\epsilon B_{n}$ for all $x'\in \delta B_{n}$, where $B_n$ denotes the unit ball in $\mathbb{R}^n$ \cite{clarke1990optimization,rockafellar1998variational}.} \cite[Proposition 2.1.5]{clarke1990optimization},
$\begin{pmatrix}
G_x(\cdot, \cdot)\\
G_y(\cdot, \cdot)
\end{pmatrix}$
and
$
\begin{pmatrix}
G_x^N(\cdot, \cdot)\\
G_y^N(\cdot, \cdot)
\end{pmatrix}
$
are osc. By \cite[Lemma 4.2 (i)]{xu2010uniform}, for any $\epsilon>0$, there exists a $\delta>0$ such that if
\begin{equation}\label{eq:G1G1N-delta}
\sup_{(x_1, y_1)\in X_1\times Y_1}\mathbb{D}\left(\begin{pmatrix}
G_x^N(x_1, y_1)\\
G_y^N(x_1, y_1)
\end{pmatrix},
\begin{pmatrix}
G_x(x_1, y_1)\\
G_y(x_1, y_1)
\end{pmatrix}
\right)\leq \delta,
\end{equation}
then $\mathbb{D}(S_N, S^*)\leq \epsilon$, where $S_N$ is the solution set of the SAA problem \eqref{eq:kkt-SAA}.

Then we only need to show
that, for any $\delta>0$, \eqref{eq:G1G1N-delta} holds for $N$ sufficiently large. Note that for any set $A, B, C\subset \mathbb{R}^{n_1+m_1}$, $\mathbb{D}(A+C, B+C)\leq \mathbb{D}(A, B)$, then
\begin{align}\label{eq:equal}
\mathbb{D}\left(\begin{pmatrix}
G_x^N(x_1, y_1)\\
G_y^N(x_1, y_1)
\end{pmatrix},
\begin{pmatrix}
G_x(x_1, y_1)\\
G_y(x_1, y_1)
\end{pmatrix}
\right) &\leq \left\|
\begin{pmatrix}
\nabla_{x_1}\Phi_{N}(x_1, y_1)-\nabla_{x_1}\bbe\left[\psi_2(x_1, y_1, \xi)\right]\\
\nabla_{y_1}\Phi_{N}(x_1, y_1)-\nabla_{y_1}\bbe\left[\psi_2(x_1, y_1, \xi)\right]
\end{pmatrix}
\right\|\nonumber\\
&\leq \left\|
\begin{pmatrix}
\nabla_{x_1}\Phi_{N}(x_1, y_1)-\nabla_{x_1}\bbe\left[\psi_2(x_1, y_1, \xi)\right]\\
\nabla_{y_1}\Phi_{N}(x_1, y_1)-\nabla_{y_1}\bbe\left[\psi_2(x_1, y_1, \xi)\right]
\end{pmatrix}
\right\|_1.\nonumber
\end{align}
Moreover, by Lemma \ref{l:2stage-bound-continuity}, $\psi_2(\cdot, \cdot, \xi)$ is continuously differentiable, then we have $\nabla_{x_1}\bbe\left[\psi_2(x_1, y_1, \xi)\right]=\bbe\left[\nabla_{x_1}\psi_2(x_1, y_1, \xi)\right]$, $\nabla_{y_1}\bbe\left[\psi_2(x_1, y_1, \xi)\right]=\bbe\left[\nabla_{y_1}\psi_2(x_1, y_1, \xi)\right]$, $\nabla_{x_1} \psi_2(\cdot, \cdot, \xi)=T(\xi)^\top {\bm \pi}_{x_2}(\cdot, \cdot, \xi)$ and $\nabla_{y_1} \psi_2(\cdot, \cdot, \xi)=-A(\xi)^\top {\bm \pi}_{y_2}(\cdot, \cdot, \xi)$ are Lipschitz continuous  over the compact set $X_1\times Y_1$. By the uniform law of large numbers \cite[Theorem 9.60]{shapiro2021lectures},
\[
\lim_{N\to\infty}\sup_{\substack{(x_1, y_1)\in X_1\times Y_1\\i=1, \cdots, n_1+m_1}}\left|
\begin{pmatrix}
\nabla_{x_1}\Phi_{N}(x_1, y_1)-\nabla_{x_1}\bbe\left[\psi_2(x_1, y_1, \xi)\right]\\
\nabla_{y_1}\Phi_{N}(x_1, y_1)-\nabla_{y_1}\bbe\left[\psi_2(x_1, y_1, \xi)\right]
\end{pmatrix}_i
\right| =0,\; a.s.
\]
and then
\[
\lim_{N\to\infty}\sup_{(x_1, y_1)\in X_1\times Y_1}\mathbb{D}\left(\begin{pmatrix}
G_x^N(x_1, y_1)\\
G_y^N(x_1, y_1)
\end{pmatrix},
\begin{pmatrix}
G_x(x_1, y_1)\\
G_y(x_1, y_1)
\end{pmatrix}
\right)=0, \; a.s.,
\]
which implies \eqref{eq:asconvergence}.
 \hfill $\square$
\end{proof}

\begin{remark}
Under the suitable assumptions of the moment generating functions of $\nabla_{x_1}\psi_2$, $\nabla_{y_1}\psi_2$ and their Lipschitz coefficients,  we also can derive the exponential convergence rate of the SAA problem \eqref{eq:ts-minimax-SAA} to the true problem \eqref{eq:ts-minimax}-\eqref{eq:st-minimax} in terms of their KKT point sets, based on the uniform large deviation results in \cite[Theorem 5.1]{shapiro2008stochastic} and \cite[Lemma 4.2 (i)]{xu2010uniform}.
\end{remark}

\section{The inexact parallel proximal gradient descent ascent algorithm for two-stage stochastic minimax problem}

This section presents IPPGDA algorithm for the SAA problem \eqref{eq:ts-minimax-SAA} of the  two-stage stochastic minimax problem \eqref{eq:ts-minimax}-\eqref{eq:st-minimax}. Let $\bar{f} := f +\bm{1}_{X_1}$ and $\bar{g} := g +\bm{1}_{Y_1}$,  where $\bm{1}$ denotes the indicator function.

 \begin{algorithm}[H]
		\caption{IPPGDA algorithm for the two-stage stochastic minimax problem \eqref{eq:ts-minimax}-\eqref{eq:st-minimax}}
		\begin{algorithmic}[1]
			\Require initial point $(x^1_1, y^1_1)$, stepsizes $\beta_1^y$, $\beta_1^x$ and sequence $\{\epsilon^k\}$
              \For{$k=1, 2, \cdots$}
              \For{$i=1, 2, \cdots, N$}
						\State Solve \eqref{eq:nonsmooth-equation} with given $(x_{1}^{k}, y_{1}^{k})$ and $\xi^{i}$ to obtain $\mu^{k,i}:=(\tilde{x}_{2}^{k,i}, \tilde{y}_{2}^{k,i}, \tilde{\pi}_{x_{2}^{k}}^{i}, \tilde{\pi}_{y_{2}^{k}}^{i})$, such that $\| H(\mu^{k,i}, \xi^i) \| < \epsilon^{k}$.
			            \EndFor
                    \State $\tilde{v}_x^k= \frac{1}{N}\sum_{i=1}^N T(\xi^i)^\top \tilde{\pi}_{x_2^k}^{i}$
                    \State$ \tilde{v}_y^k = -\frac{1}{N}\sum_{i=1}^N A(\xi^i)^\top\tilde{\pi}_{y_2^k}^{i}$
              \State $y_1^{k+1} = \arg\max_{y_1\in \mathbb{R}^{m_1}} \{\langle \nabla_{y_1} \psi_1(x_1^k, y_1^k) + \tilde{v}_y^k, y_1-y_1^{k}\rangle-\bar{g}(y_1) - \frac{1}{2\beta_1^y}\|y_1-y_1^{k}\|^2\}$
              \State $x_1^{k+1} = \arg\min_{x_1\in \mathbb{R}^{n_1}} \{\bar{f}(x_1)  +\langle \nabla_{x_1}\psi_1(x_1^k, y_1^k) + \tilde{v}_x^k, x_1-x_1^{k}\rangle + \frac{1}{2\beta_1^x}\|x_1-x_1^{k}\|^2\}$
              \EndFor
		\end{algorithmic}
	\end{algorithm}
	
\begin{remark}\label{r:delta}
IPPGDA algorithm is an inexact version of   PPGDA in \cite{cohen2025alternating}. Since the gradient of the second-stage minimax value function $\Phi_{N}(x_1, y_1)$ cannot be computed exactly, we cannot apply PPGDA. Instead, IPPGDA algorithm finds an inexact solution of problem \eqref{eq:nonsmooth-equation}, which defines an inexact
gradient of  $\Phi_{N}(x_1, y_1)$ at each step.
 In Section 5.4, we will show that under suitable conditions,
$\| H(\mu^{k,i}, \xi) \| <  \frac{\delta^{k} \sqrt{\underline{\lambda}}}{\max\{\bar{a}, \bar{t}\}}$ implies $\|\tilde{\pi}_{x_2^k}^{i} - \pi_{x_2^k} \|\leq \frac{\delta^k}{\bar{a}}$ and $\|\tilde{\pi}_{y_2^k}^{i} - \pi_{y_2^k}\|\leq \frac{\delta^k}{\bar{t}}$, where $\bar{a}:=\max_{\xi\in\Xi}\|A(\xi)\|$, $\bar{t}:=\max_{\xi\in\Xi}\|T(\xi)\|$, $(\pi_{x_2^k}^i, \pi_{y_2^k}^i)$ is the $\pi$-component of the unique solution to \eqref{eq:nonsmooth-equation}. Let
$$\delta^k_x:=\tilde{v}_x^k- v_x^k, \quad \quad  \delta^k_y:=\tilde{v}_y^k- v_y^k,$$
where $v_x^k$ and $v_y^k$ are the true gradients of $\Phi_{N}(x_1, y_1)$ w.r.t. $x_1$ and $y_1$.  Consequently, $\|\delta^k_x\|=\|\tilde{v}_x^k- v_x^k\|\leq\delta^k$  and $\|\delta^k_y\|=\|\tilde{v}_y^k- v_y^k\|\leq\delta^k$. The requirement of $\{\delta^k\}$ for convergence of IPPGDA algorithm will be presented in Lemma \ref{l:condition1}.
\end{remark}

We  consider the SAA two-stage stochastic minimax problem  \eqref{eq:ts-minimax-SAA} as the following minimization problem
\begin{equation}\label{eq:onemin}
\min_{x_1} \; \Psi_N(x_1) := \bar{f}(x_1) + \theta_N(x_1),
\end{equation}
where
\[
\theta_N(x_1):=\max_{y_1} ~\psi_1(x_1, y_1) + \Phi_{N}(x_1, y_1) -\bar{g}(y_1),
\]
and prove that the subsequence and global convergence of IPPGDA algorithm to a critical point of problem \eqref{eq:onemin}. As in \cite{cohen2025alternating}, we consider the perturbed gradient-like descent sequence as follows:
\begin{definition}\label{d:pglds}
A sequence $\{(x_1^k, \nu_k)\}_{k\in\mathbb{N}}\subset \dom~ \Psi_N \times \mathbb{R}_+$ is called a perturbed gradient-like descent sequence if the following conditions hold.
\begin{itemize}
\item Condition 1 (Perturbed sufficient decrease property). There exists $c_1>0$ such that for every $k\in\mathbb{N}$,
\begin{equation}\label{eq:psd}
c_1(\|x_1^{k+1} - x_1^k\| + \nu_k^2)\leq \left(\Psi_N(x_1^k)+\frac{1}{2}\nu_k^2\right)- \left(\Psi_N(x_1^{k+1})+\frac{1}{2}\nu_{k+1}^2\right).
\end{equation}

\item Condition 2 (Perturbed subgradient lower bound on iterates gap). There exists $c_2>0$ such that for every $k\in\mathbb{N}$, one can find $\zeta^{k+1}\in\partial \Psi_N(x_1^{k+1})$, which satisfies
\begin{equation}\label{eq:psl}
\|\zeta^{k+1}\|\leq c_2(\|x_1^{k}-x_1^{k+1}\| + \nu_k).
\end{equation}

\item Condition 3. Let $\{x_1^k\}_{k\in\mathcal{K}\subset\mathbb{N}}$ be a subsequence that converges to a point $\bar{x}_1^N$. Then
\[
\lim\sup_{k\in\mathcal{K}\subset\mathbb{N}} \; \Psi_N(x_1^k)\leq \Psi_N(\bar{x}_1^N).
\]
\end{itemize}
\end{definition}
The focus of this section is on how to incorporate the error terms $\|\tilde{v}_x^k- v_x^k\|$ and  $\|\tilde{v}_y^k- v_y^k\|$ into the perturbed gradient-like descent sequence, satisfying the three conditions in Definition \ref{d:pglds}, without increasing the computational burden.


In Section 5.1, we establish key  estimates for the sequences $\{(x_1^k, y_1^k)\}$. In Section 5.2, we verify that the generated sequence $\{(x_1^k, \nu_k)\}_{k\in\mathbb{N}}$ is a perturbed gradient-like descent sequence.  Section 5.3 investigates the Kurdyka–Łojasiewicz (KL) properties of the problem and the subsequence and global convergence of IPPGDA algorithm to a critical point of \eqref{eq:onemin}. Finally, in Section 5.4, we address the second-stage problem with an inexact solution.

\subsection{The key estimates of IPPGDA algorithm}

We consider Lipschitz continuity properties of problem (\ref{eq:onemin}) in the following lemma. To this end, in addition to the Lipschitz continuity of $\nabla \psi_1$, we further assume that there exist $L^1_{xx}>0$, $L^1_{xy}>0$, $L^1_{yx}>0$, $L^1_{yy}>0$  such that for any $x_1, \bar{x}_1\in\mathbb{R}^{n_1}$ and $y_1,\bar{y}_1\in\mathbb{R}^{m_1}$,
    \[
    \begin{matrix}
    \|\nabla_{x_1}\psi_1(x_1, y_1)-\nabla_{x_1}\psi_1(\bar{x}_1, y_1)\|\leq L^1_{xx}\|x_1-\bar{x}_1\|,\\
    \|\nabla_{x_1}\psi_1(x_1, y_1)-\nabla_{x_1}\psi_1(x_1, \bar{y}_1)\|\leq L^1_{xy}\|y_1-\bar{y}_1\|,\\
    \|\nabla_{y_1}\psi_1(x_1, y_1)-\nabla_{y_1}\psi_1(\bar{x}_1, y_1)\|\leq L^1_{yx}\|x_1-\bar{x}_1\|,\\
    \|\nabla_{y_1}\psi_1(x_1, y_1)-\nabla_{\y_1}\psi_1(x_1, \bar{y}_1)\|\leq L^1_{yy}\|y_1-\bar{y}_1\|.
    \end{matrix}
    \]

\begin{lemma}\label{l:properties}
Let
\begin{equation}\label{eq:y-max}
\tilde{y}_1^N(x_1):=\arg\max_{y_1\in Y_1} \psi_1(x_1, y_1) + \Phi_{N}(x_1, y_1)-\bar{g}(y_1).
\end{equation}
Under Assumption \ref{a:second-stage}, the following statements hold.

(i) There exists $L_{yx}>0$ such that $\tilde{y}_1^N:\mathbb{R}^{n_1}\to\mathbb{R}^{m_1}$ is $(L_{yx}/\sigma)$-Lipschitz continuous.

(ii) $\nabla \theta_N(x_1) = \nabla_{x_1} \psi_1(x_1, \tilde{y}_1^N(x_1)) + \nabla_{x_1}\Phi_{N}(x_1, \tilde{y}_1^N(x_1))$.

(iii) There exists $L_{\theta}>0$ such that $\nabla \theta_N(\cdot)$ is $L_{\theta}$-Lipschitz continuous.
\end{lemma}

\begin{proof}
By Lemma \ref{l:2stage-bound-continuity} (iii),  both
\[
\nabla_{x_1}\Phi_{N}(\cdot, \cdot) =\frac{1}{N}\sum_{i=1}^NT(\xi^i)^\top {\bm \pi}_{x_2}(\cdot, \cdot, \xi^i)
\]
and
\[
\nabla_{y_1} \Phi_{N}(\cdot, \cdot)
=\frac{1}{N}\sum_{i=1}^N-A(\xi^i)^\top {\bm \pi}_{y_2}(\cdot, \cdot, \xi^i)
\]
are Lipschitz continuous  over $X_1\times Y_1$. Then there exist $L^2_{xx}>0$, $L^2_{xy}>0$, $L^2_{yx}>0$ and $L^2_{yy}>0$  such that for any $x_1, \bar{x}_1\in X_1$ and $y_1,\bar{y}_1\in Y_1$,
    \[
    \begin{matrix}
    \| \nabla_{x_1}\Phi_{N}(x_1, y_1)  -\nabla_{x_1}\Phi_{N}(\bar{x}_1, y_1) \|\leq L^2_{xx}\|x_1-\bar{x}_1\|,\\
    \| \nabla_{x_1}\Phi_{N}(x_1, y_1)  -\nabla_{x_1}\Phi_{N}(x_1, \bar{y}_1) \|\leq L^2_{xy}\|y_1-\bar{y}_1\|,\\
     \| \nabla_{y_1}\Phi_{N}(x_1, y_1)  -\nabla_{y_1}\Phi_{N}(\bar{x}_1, y_1) \|\leq L^2_{yx}\|x_1-\bar{x}_1\|,\\
   \| \nabla_{y_1}\Phi_{N}(x_1, y_1)  -\nabla_{y_1}\Phi_{N}(x_1, \bar{y}_1) \| \leq L^2_{yy}\|y_1-\bar{y}_1\|.
    \end{matrix}
    \]
 Let $L_{xx}:=L_{xx}^1+L_{xx}^2$,    $L_{xy}:=L_{xy}^1+L_{xy}^2$, $L_{yx}:=L_{yx}^1+L_{yx}^2$, $L_{yy}:=L_{yy}^1+L_{yy}^2$ and $L_\theta=L_{xx}+\frac{L_{xy}L_{yx}}{\sigma}$.
    Then (i) is  from \cite[Lemma 1]{cohen2025alternating}, (ii) is from \cite[Proposition 1]{cohen2025alternating} and (iii) is  from \cite[Lemma 2]{cohen2025alternating}.
 \hfill $\square$
\end{proof}

Then we analyze the $y$-step (step 7) and establish the inexact relationship between $y_1^{k+1}$ and $y_1^k$, and the corresponding maximizers $\tilde{y}_1^N(x_1^k)$ and $\tilde{y}_1^N(x_1^{k+1})$.

\begin{lemma}\label{l:y-step}
Suppose that Assumption \ref{a:second-stage} holds. Let $\kappa=L_{yy}/\sigma$ and $\epsilon_y^k=\frac{2\|\delta_y^k\|}{\sigma + L_{yy}}.$  Then for any $\beta>0$ and every $k\in\mathbb{N}$, we have
\begin{equation}\label{eq:l1}
\|y_1^{k+1} - \tilde{y}_1^N(x_1^{k+1})\| \leq \sqrt{\kappa/(\kappa+1)} \|y_1^k - \tilde{y}_1^N(x_1^k)\| + \frac{L_{yx}}{\sigma} \|x_1^{k+1} - x_1^k\|
 + \epsilon^k_y,
\end{equation}

\begin{align}\label{eq:l2}
\|y_1^{k+1} - \tilde{y}_1^N(x_1^k)\| \leq & \sqrt{\kappa/(\kappa+1)}
\left( \|y_1^k - \tilde{y}_1^N(x_1^{k-1})\| + \frac{L_{yx}}{\sigma} \|x_1^k - x_1^{k-1}\| \right)  + \epsilon_y^k,
\end{align}

\begin{align}\label{eq:l3}
\|y_1^{k+1} - \tilde{y}_1^N(x_1^{k+1})\|^2
&\leq (1+\beta) \Big(\frac{\kappa + 1/2}{\kappa + 1} \|y_1^k - \tilde{y}_1^N(x_1^k)\|^2 + \frac{(2\kappa + 1)L_{yx}^2}{\sigma^2} \|x_1^{k+1} - x_1^k\|^2\Big)\nonumber \\
&\quad + (1+\beta^{-1})(\epsilon_y^k)^2,
\end{align}

\begin{align}\label{eq:l4}
\|y_1^{k+1} - \tilde{y}_1^N(x^k)\|^2 &\leq \displaystyle{(1+\beta)\frac{\kappa + 1/2}{\kappa + 1} \left( \|y_1^k - \tilde{y}_1^N(x_1^{k-1})\|^2 + \frac{2\kappa L_{yx}^2}{\sigma^2} \|x_1^k - x_1^{k-1}\|^2 \right)}\nonumber\\
&\quad +\displaystyle{(1+\beta^{-1})(\epsilon_y^k)^2}
\end{align}
\end{lemma}

\begin{proof}
Consider the following problem:
\begin{equation}\label{eq:l311}
\min_{y_1\in\mathbb{R}^{m_1}} \; \Gamma_k(y_1):=\bar{g}(y_1) - \psi_1(x_1^k, y_1) -  \Phi_{N}(x_1^k, y_1).
\end{equation}
Then $\tilde{y}_1^N(x_1^k)=\argmin_{y_1\in \mathbb{R}^{m_1}} \Gamma_k(y_1)$ is unique. By $\sigma$-strong convexity of $\Gamma_k$,  it follows that
\[
\Gamma_k(y_1^{k+1}) - \Gamma_k(\tilde{y}_1^N(x_1^k))\geq \langle 0, y_1^{k+1} -  \tilde{y}_1^N(x_1^k)\rangle +\frac{\sigma}{2}\|y_1^{k+1} -  \tilde{y}_1^N(x_1^k)\|^2.
\]
Note that $y_1^{k+1}\in Y_1$ and $\tilde{y}_1^N(x_1^k)\in Y_1$.  Applying \cite[Lemma 3.1]{TeboulleM18} to $\bar{g}(y_1) - \langle \nabla_{y_1} \psi_1(x_1^k, y_1^k) + v_y^k+\delta_y^k, y_1\rangle$ yields:
\begin{equation}\label{eq:l31}
\begin{aligned}
&\frac{2}{L_{yy}}\left[\bar{g}(y_1^{k+1}) -\bar{g}(\tilde{y}_1^N(x_1^k))- \langle \nabla_{y_1} \psi_1(x_1^k, y_1^k) + v_y^k+\delta_y^k, y_1^{k+1}-\tilde{y}_1^N(x_1^k)\rangle\right]\\
&\leq \|y_1^k - \tilde{y}_1^N(x_1^k)\|^2-\|y_1^{k+1} - \tilde{y}_1^N(x_1^k)\|^2-\|y_1^{k+1} - y_1^k\|^2.
\end{aligned}
\end{equation}
Applying \cite[Lemma 3.3]{TeboulleM18} to $\psi_1(x_1^k, y_1) + \Phi_{N}(x_1^k, y_1)$ gives:
\begin{equation}\label{eq:l33}
\begin{aligned}
&\psi_1(x_1^k, \tilde{y}_1^N(x_1^k)) + \Phi_{N}(x_1^k, \tilde{y}_1^N(x_1^k))-\left(\psi_1(x_1^k, y_1^{k+1}) + \Phi_{N}(x_1^k, y_1^{k+1})\right)\\
&\leq \langle -\nabla_{y_1}\psi_1(x_1^k, y_1^k) -v_y^k, y_1^{k+1}-\tilde{y}_1^N(x_1^k) \rangle.
\end{aligned}
\end{equation}
Combining \eqref{eq:l311}-\eqref{eq:l33}, we obtain:
\[
\begin{aligned}
&\frac{L_{yy}}{2}\left(\|y_1^k - \tilde{y}_1^N(x_1^k)\|^2-\|y_1^{k+1} - \tilde{y}_1^N(x_1^k)\|^2\right)  + \|\delta_y^k\|\|y_1^{k+1} - \tilde{y}_1^N(x_1^k)\|\\
&\geq\Gamma_k(y_1^{k+1}) - \Gamma_k(\tilde{y}_1^N(x_1^k))\\
&\geq \frac{\sigma}{2} \|y_1^{k+1} - \tilde{y}_1^N(x_1^k)\|^2.
\end{aligned}
\]
Consequently,
\[
\|y_1^{k+1} - \tilde{y}_1^N(x_1^k)\|^2\leq \frac{L_{yy}}{\sigma + L_{yy}}\|y_1^{k} - \tilde{y}_1^N(x_1^k)\|^2 + \epsilon_y^k\|y_1^{k+1} - \tilde{y}_1^N(x_1^k)\|,
\]
which implies 
\[
\left(\|y_1^{k+1} - \tilde{y}_1^N(x_1^k)\|-\frac{1}{2}\epsilon_y^k\right)^2\leq \left(\frac{\sqrt{L_{yy}}}{\sqrt{\sigma + L_{yy}}}\|y_1^{k} - \tilde{y}_1^N(x_1^k)\|+ \frac{1}{2}\epsilon_y^k\right)^2.
\]
Then the above inequality implies
\[
\|y_1^{k+1} - \tilde{y}_1^N(x_1^k)\|\leq \sqrt{\kappa/(\kappa+1)} \|y_1^{k} - \tilde{y}_1^N(x_1^k)\| + \epsilon_y^k.
\]
Combining the triangle inequality, the Lipschitz continuity of $\tilde{y}_1^N(\cdot)$ (Lemma \ref{l:properties} (i)) and the above inequality, we have \eqref{eq:l1} and \eqref{eq:l2} as follows:
\begin{align*}
\|y_1^{k+1} - \tilde{y}_1^N(x_1^{k+1})\|
&\leq \|y_1^{k+1} - \tilde{y}_1^N(x_1^k)\| + \|\tilde{y}_1^N(x_1^k) - \tilde{y}_1^N(x_1^{k+1})\| \nonumber \\
&\leq \sqrt{\frac{\kappa}{\kappa+1}} \|y_1^{k} - \tilde{y}_1^N(x_1^k)\| + \frac{L_{yx}}{\sigma} \|x_1^{k+1} - x_1^k\|+ \epsilon_y^k \nonumber
\end{align*}
and
\begin{align*}
\|y_1^{k+1} - \tilde{y}_1^N(x_1^k)\|
&\leq \sqrt{\frac{\kappa}{\kappa+1}} \|y_1^k - \tilde{y}_1^N(x_1^k)\| + \epsilon_y^k \nonumber \\
&\leq \sqrt{\frac{\kappa}{\kappa+1}} \Bigl( \|y_1^k - \tilde{y}_1^N(x_1^{k-1})\|
+ \|\tilde{y}_1^N(x_1^k) - \tilde{y}_1^N(x_1^{k-1})\| \Bigr) + \epsilon_y^k \nonumber \\
&\leq \sqrt{\frac{\kappa}{\kappa+1}} \left( \|y_1^k - \tilde{y}_1^N(x_1^{k-1})\| + \frac{L_{yx}}{\sigma} \|x_1^k - x_1^{k-1}\| \right)  + \epsilon_y^k.
\end{align*}
Finally,  squaring both sides of \eqref{eq:l1} and \eqref{eq:l2}, and applying $(a+b)^2\leq (1+\gamma)a^2 + (1+\gamma^{-1})b^2$ twice with $\gamma=1/2\kappa$ and $\beta$, respectively, we obtain \eqref{eq:l3} and \eqref{eq:l4}.
 \hfill $\square$
\end{proof}

Next we consider the $x$-step (step 8) from view of function value gap and subgradient bound in the following lemma.

\begin{lemma}\label{l:x-step}
Suppose that Assumption \ref{a:second-stage} holds. Let $\{(x^k, y^k)\}_{k\in \mathbb{N}}$ be the sequence generated by IPPGDA algorithm. Then
\begin{itemize}
\item[(i)] 
for every $k\geq0$, we have
\begin{align}\label{eq:446}
\Psi_N(x_1^{k+1}) - \Psi_N(x_1^k) &\leq -\frac{1}{2} \left( \frac{1}{\beta_1^x} - L_{xy}^2 - L_{\theta}-1 \right) \|x_1^{k+1} - x_1^k\|^2 \nonumber\\
&+ \frac{1}{2} \|y_1^k - \tilde{y}_1^N(x_1^k)\|^2+\frac{1}{2}\|\delta_x^k\|^2;
\end{align}

\item[(ii)] there exist $M>0$ and $\zeta^{k+1}\in\partial \Psi_N(x^{k+1})$, which satisfy
\begin{equation}\label{eq:447}
\|\zeta^{k+1}\|\leq M(\|x^{k+1} - x^{k}\| + \|y^{k} - \tilde{y}_1^N(x^k)\|)+\|\delta_x^k\|.
\end{equation}
\end{itemize}
\end{lemma}

\begin{proof}
The proof of the lemma is similar as \cite[Lemmas 7 and 8]{cohen2025alternating}.
Note that the key difference is that there is a $\delta_x^k$ in  $x$-step (step 8) as follows
\[
x_1^{k+1} = \argmin_{x_1\in \mathbb{R}^{n_1}} \left\{\bar{f}(x_1)  +\langle \nabla_{x_1}\psi_1(x_1^k, y_1^k) + v_x^k+\delta_x^k, x_1-x_1^{k}\rangle + \frac{1}{2\beta_1^x}\|x_1-x_1^{k}\|^2\right\}.
\]
This leads to the $\frac{1}{2}\|\delta_x^k\|^2$ in \eqref{eq:446} and $\|\delta_x^k\|$ in \eqref{eq:447}. Since the remainder of the proof is essentially the same as in \cite[Lemmas 7 and 8]{cohen2025alternating}, we omit the details of the proof.
 \hfill $\square$
\end{proof}

\begin{remark}
Lemmas \ref{l:y-step} and \ref{l:x-step} are extensions of \cite[Lemmas 6, 7, and 8]{cohen2025alternating}. In \cite{cohen2025alternating}, the authors analyzed the divergence caused by using an inexact proximal gradient step:
\[
y_1^{k+1} = \argmax_{y_1\in \mathbb{R}^{m_1}} \left\{\left\langle \nabla_{y_1} \psi_1(x_1^k, y_1^k) + v_y^k, y_1-y_1^{k}\right\rangle-\bar{g}(y_1) - \frac{1}{2\beta_1^y}\|y_1-y_1^{k}\|^2\right\}
\]
 to approximate the inner maximization \eqref{eq:y-max}. Different from the PPGDA in \cite{cohen2025alternating}, we can not calculate $v_y^k$ in step 7  and $v_x^k$ in step 8 of Algorithm 1 exactly. Instead, we do inexact  proximal gradient steps with inexact terms  $\tilde{v}_x^k$ and $\tilde{v}_y^k$.
\end{remark}

\subsection{Perturbed gradient-like descent sequence}
In this section, we show that the sequence $\{(x_1^k, \nu_k)\}$ is a perturbed gradient-like descent sequence,  where $\nu_k$ is defined as
\begin{equation}\label{eq:nuk}
\nu_k:=\sqrt{s\|y_1^k-\tilde{y}_1^N(x^k)\|^2+(\delta^{k-1})^2},
\end{equation}
where $s>0$,  $\delta^k\geq\max\{\|\delta_x^k\|,\|\delta_y^k\|\}$ and $\{\delta^k\}\downarrow0$.
 We consider the perturbed sufficient descent property firstly. 

\begin{lemma}\label{l:condition1}
 Let $\{(x_1^k, y_1^k)\}_{k\in\mathbb{N}}$ be the sequence generated by Algorithm 1. Suppose that Assumption \ref{a:second-stage} holds and there is
$\beta\in(0, \frac{1}{2\kappa+1})$ such that
\[
\frac{1}{\beta_1^x} \geq L_{xy}^2+L_{\theta} +1 +\max\left\{\bar{\eta}, \frac{(2\kappa+2)\bar{\eta}}{(1-\beta-2\beta\kappa)}\right\}>0,
\]
\[
s\in \left[\frac{2\kappa+2}{1-\beta-2\beta\kappa}, \frac{(\frac{1}{\beta_1^x} - L_{xy}^2 - L_{\theta}-1)}{\bar{\eta}}\right],
\]
and
\begin{equation}\label{eq:l4-01}
(\delta^{k-1})^2\geq \frac{\kappa+1}{(1+\beta)(\kappa+1/2)}\left(\frac{4s(1+\beta)}{\beta(\sigma +L_{yy})^2}+2\right)(\delta^k)^2,
\end{equation}
where $\bar{\eta}:=\frac{(1+\beta)(2\kappa+1)L_{yx}^2}{\sigma^2}$.
Then there exists $c_1>0$ such that for every $k\in\mathbb{N}$, \eqref{eq:psd} (Condition~1 in Definition \ref{d:pglds}) holds.
\end{lemma}

\begin{proof}
For every $k\in\mathbb{N}$, let
\[
\Delta_k:=\left(\Psi_N(x_1^{k+1})+\frac{1}{2}\nu_{k+1}^2\right)-\left(\Psi_N(x_1^k)+\frac{1}{2}\nu_k^2\right).
\]
Then by Lemma \ref{l:x-step} (i), we have
\begin{align}\label{eq:l4-1}
\Delta_k\leq \frac{1}{2}\left(   L_{xy}^2 + L_{\theta}+1 -\frac{1}{\beta_1^x}\right) \|x_1^{k+1} - x_1^k\|^2 +\frac{1}{2}\|\delta_x^k\|^2 + \frac{1}{2s} \nu_k^2+\frac{1}{2}(\nu_{k+1}^2-\nu_k^2).
\end{align}
Moreover, by applying \eqref{eq:l3} in Lemma \ref{l:y-step}, we have
\begin{equation}\label{eq:l4-2}
\begin{array}{lll}
\nu_{k+1}^2+\|\delta_x^k\|^2&=& s\|y_1^{k+1}-\tilde{y}_1^N(x^{k+1})\|^2+\|\delta_y^{k}\|^2 +\|\delta_x^k\|^2\\
&\leq&\displaystyle{\frac{(1+\beta)(\kappa+1/2)}{\kappa+1}s\|y_1^k-\tilde{y}_1^N(x_1^k)\|^2}+\displaystyle{\left(\frac{4s(1+\beta)}{\beta(\sigma +L_{yy})^2}+2\right)(\delta^{k})^2}\\
&+& \displaystyle{\bar{\eta}s\|x_1^k-x_1^{k+1}\|^2}.
\end{array}
\end{equation}
Substituting $(\delta^{k-1})^2$ (in \eqref{eq:l4-01}) into   \eqref{eq:l4-2}, we have
\begin{equation}\label{eq:l4-3}
\begin{array}{lll}
\nu_{k+1}^2 &\leq& \displaystyle{\frac{(1+\beta)(\kappa+1/2)}{\kappa+1}\left(s\|y_1^k-\tilde{y}_1^N(x_1^k)\|^2+(\delta^{k-1})^2 \right) + \bar{\eta}s\|x_1^k-x_1^{k+1}\|^2}\\
&=& \displaystyle{\bar{\eta}s\|x_1^k-x_1^{k+1}\|^2}.
\end{array}
\end{equation}
Combining \eqref{eq:l4-1} and \eqref{eq:l4-3}, we obtain
\[
\Delta_k\leq -\frac{t_1}{2}\|x_1^{k+1} - x_1^k\|^2-\frac{t_2}{2}\nu_k^2
\]
with
\[
t_1=\frac{1}{\beta_1^x} - L_{xy}^2 - L_{\theta}-1 - \bar{\eta}s
\]
and
\[ t_2=1-\frac{1}{s}-\frac{(1+\beta)(\kappa+1/2)}{\kappa+1}=\frac{s(\frac{1-\beta}{2}-\beta\kappa) - \kappa-1}{s(\kappa+1)}.
\]
Since $\beta<\frac{1}{1+2\kappa}$ and $s>\frac{2\kappa+2}{1-\beta-2\beta\kappa}>0$, we have $t_2>0$. Moreover, since there exists sufficiently small $\beta_1^x>0$ such that
\[
\frac{1}{\beta_1^x} \geq L_{xy}^2+L_{\theta} +1 +\max\left\{\bar{\eta}, \frac{(2\kappa+2)\bar{\eta}}{(1-\beta-2\beta\kappa)}\right\}>0,
\]
we have $t_1>0$, $\frac{2\kappa+2}{1-\beta-2\beta\kappa} < \frac{(\frac{1}{\beta_1^x} - L_{xy}^2 - L_{\theta}-1)}{\bar{\eta}}$ and $s\in [\frac{2\kappa+2}{1-\beta-2\beta\kappa}, \frac{(\frac{1}{\beta_1^x} - L_{xy}^2 - L_{\theta}-1)}{\bar{\eta}}]$. Then we obtain \eqref{eq:psd} (Condition 1) with $c_1=\min\{t_1, t_2\}/2$.
 \hfill $\square$
\end{proof}

We then consider perturbed subgradient lower bound property (Condition 2 in Definition \ref{d:pglds}).

\begin{lemma}[Perturbed subgradient lower bound]\label{l:condition2}
Suppose that Assumption \ref{a:second-stage} holds. Let $\{(x_1^k, y_1^k)\}_{k\in\mathbb{N}}$ be the sequence generated by Algorithm 1, and $\{\nu_k\}_{k\in\mathbb{N}}$ be defined in \eqref{eq:nuk}. Then, there exists $c_2>0$ such that for every $k\in\mathbb{N}$, one can find $\zeta^{k+1}\in\partial\Psi_N(x_1^{k+1})$, which satisfies \eqref{eq:psl} (Condition 2 in Definition \ref{d:pglds}).
\end{lemma}

\begin{proof}
By Lemma \ref{l:x-step} (ii), there exists $M>0$ such that for  every $k>0$ and $c_2=\max\left\{M, \sqrt{\max\left\{2M^2/s, 2\right\}}\right\}$, one can find $\zeta^{k+1}\in\partial\Psi_N(x_1^{k+1})$ such that
\begin{equation*}
\begin{array}{lll}
\|\zeta^{k+1}\|&\leq& M(\|x^{k+1} - x^{k}\| + \|y^{k} - \tilde{y}_1^N(x^k)\|)+\|\delta_x^k\|\\
&\leq&M(\|x^{k+1} - x^{k}\|) +\sqrt{2M^2\|y^{k} - \tilde{y}_1^N(x^k)\|^2+2(\delta^{k-1})^2}\\
&\leq&M(\|x^{k+1} - x^{k}\|) + \sqrt{\max\left\{\frac{2M^2}{s}, 2\right\}}\nu_k\\
&\leq&c_2(\|x^{k+1} - x^{k}\| +\nu_k),
\end{array}
\end{equation*}
where the second inequality follows from $(a+b)^2\leq (1+\gamma)a^2+(1+\gamma^{-1})b^2$ and $\|\delta_x^k\|\leq\delta^k\leq\delta^{k-1}$ in \eqref{eq:l4-01}.
 \hfill $\square$
\end{proof}

Then we consider Condition 3 in Definition \ref{d:pglds}.

\begin{lemma}\label{l:condition3}
Suppose that Assumption \ref{a:second-stage} holds. Let $\{(x_1^k, y_1^k)\}_{k\in\mathbb{N}}$ be the sequence generated by Algorithm 1, and $\{\nu_k\}_{k\in\mathbb{N}}$ be as defined in \eqref{eq:nuk}, such that $\{(x_1^k, \nu_k)\}_{k\in\mathbb{N}}$ satisfies Condition 1 in Definition \ref{d:pglds}. Let $\{x_1^k\}_{k\in\mathcal{K}\subset\mathbb{N}}$ be a subsequence that converges to a point $\bar{x}_1$. Then
\[
\limsup_{k\to\infty, k\in\mathcal{K}} \; \Psi_N(x_1^k)\leq \Psi_N(\bar{x}_1).
\]
\end{lemma}

\begin{proof}
The proof of the lemma is similar to that of \cite[Lemma 13]{cohen2025alternating}; the only difference is that we need to add a noise term $\delta_x^k$ after $\nabla_{x_1}\psi_1(x_1^k, y_1^k) + v_x^k$ (that is $\nabla_uc(u^l, w^l)$ in the proof of \cite[Lemma 13]{cohen2025alternating}). Note that this change does not alter the proof process, we omit the details of the proof here.
 \hfill $\square$
\end{proof}

\subsection{Subsequence and global convergence}
In the above section, we have verified three conditions in Definition \ref{d:pglds}, which implies the sequence $\{x_1^k, \nu_k\}$ is a  perturbed gradient-like descent sequence. Now we investigate the subsequence and global convergence of IPPGDA algorithm. We start from the KL property of our problem.

\begin{assumption}\label{a:semialgebraic}
    $f$, $\psi_1$, $F_2(\cdot, \cdot, \xi^i) (i=1, \cdots, N),$ and $g$ are semialgebraic.
\end{assumption}

\begin{lemma}\label{l:semia}
    Suppose that $\phi:\mathbb{R}^n\times \mathbb{R}^m\to\mathbb{R}$ is semialgebraic of two variables and a semialgebraic set $Q$ in $y$-space, then the function $\varphi(\cdot):=\inf\{\phi(\cdot, y): y\in Q\}$ is semialgebraic.
\end{lemma}

\begin{proof}
    The proof is summarized from \cite[Page 395]{ioffe2017variational}.
    Since
    \[
    \inmat{epi}~ \phi=\{(x, y, \alpha): \alpha\geq \phi(x, y), x\in \mathbb{R}^n, y\in Q\}
    \]
     is semialgbraic, $\inmat{epi}~ \varphi=\{(x, \alpha): \alpha\geq \varphi(x)\}$ and $\phi(x,y)\geq \varphi(x)$ for all $y\in Q$, we have $\inmat{epi}~ \varphi$ is the closure of the projection of $\inmat{epi}~ \phi(x, y)$ onto the $(x, \alpha)$-space.  Then by Tarski–Seidenberg theorem, $\inmat{epi}~ \varphi$ is a semialgebraic set, and then $\varphi(\cdot)$ is a semialgebraic function.
 \hfill $\square$
\end{proof}

\begin{proposition}\label{p:semialgebraic}
    Suppose that Assumptions  \ref{a:second-stage} and \ref{a:semialgebraic} hold. Then $\Psi_N$ is semialgebraic.
\end{proposition}

\begin{proof}
Note that $\bm{1}_{Y_2(y_1, \xi)}(y_2)$ and $\bm{1}_{X_2(x_1, \xi)}(x_2)$ are semialgebraic functions w.r.t. $(y_1, y_2)$ and $(x_1, x_2)$ respectively. Moreover, by \eqref{eq:f21} and \eqref{eq:st-minimax},
for any $\xi^i$, $i=1, \cdots, N$,
\[
 f_{21}(x_2,  y_1, \xi^i)=\max_{y_2}F_2(x_2, y_2, \xi^i) - \bm{1}_{Y_2(y_1, \xi^i)}(y_2)
\]
and
 \[
 \psi_2(x_1, y_1, \xi^i)=\min_{x_2}  f_{21}(x_2,  y_1, \xi^i) + \bm{1}_{X_2(x_1, \xi^i)}(x_2).
 \]
Then by  Assumption~\ref{a:semialgebraic} and Lemma~\ref{l:semia}, $ f_{21}(x_2, y_1, \xi^i)$ is semialgebraic, and then $\psi_2(x_1, y_1, \xi^i)$ is semialgebraic for every $\xi^i$, $i=1, \cdots, N$.

Since $\psi_1$, $\psi_2(\cdot, \cdot, \xi^i)$, $i=1, \cdots, N$, and $g$ are all semialgebraic, by Lemma \ref{l:semia}, $\theta_N$ in \eqref{eq:onemin} is semialgebraic, and then $\Psi_N$ is semialgebraic.
 \hfill $\square$
\end{proof}

Unlike the direct assumption of semialgebraicity for related functions in \cite{cohen2025alternating}, we must rigorously establish the semialgebraic property of function $\Psi_N$ in our framework.

Finally, we arrive the subsequence and global convergence here.

\begin{theorem}\label{t:mainconvergence}
Suppose that Assumption \ref{a:second-stage} holds. Let $\{(x_1^k, y_1^k)\}_{k\in\mathbb{N}}$ be the sequence generated by Algorithm 1. Then the following statements hold.
\begin{itemize}
\item[(i)] There exists a nonempty and bounded  set of cluster points of the sequence $\{x_1^k\}_{k\in\mathbb{N}}$, such that there exists $\Omega\subset \inmat{crit}~\Psi_N$, $\lim_{k\to\infty}\dist(x_1^k, \Omega)=0$, and $\Psi_N$ is finite and constant on $\Omega$. Moreover, let $\{x_1^k\}_{k\in\mathcal{K}\subset\mathbb{N}}$ be a subsequence converging to point $\bar{x}_1^N\in X_1$. Then the subsequence $\{y_1^k\}_{k\in\mathcal{K}\subset\mathbb{N}}$ converges to $\tilde{y}_1^N(\bar{x}_1^N)$;

\item[(ii)] In addition, suppose  Assumption \ref{a:semialgebraic} holds. Then $\sum_{k=1}^\infty\|x_1^{k+1}-x_1^k\|<\infty$ and $\{x_1^k\}_{k\in\mathbb{N}}$  converges to a critical point $\bar{x}_1^N\in\inmat{crit}~\Psi_N$. Moreover,  $\{y_1^k\}_{k\in\mathbb{N}}$ converges to $\tilde{y}_1^N(\bar{x}_1^N)$.
\end{itemize}
\end{theorem}

\begin{proof}
Since $X_1$ is compact, by Lemmas~\ref{l:condition1}-\ref{l:condition3},  the sequence $\{(x_1^k, \nu_k)\}_{k\in\mathbb{N}}$ is a gradient-like descent sequence and  the subsequence convergence of $\{x_1^k\}_{k\in\mathbb{N}}$ comes from \cite[Lemma 4]{cohen2025alternating}. Moreover, under Assumption \ref{a:semialgebraic}, by Proposition \ref{p:semialgebraic}, $\Psi_N$ is semialgebraic. Then we can apply \cite[Theorem 1]{cohen2025alternating} to obtain the global convergence of $\{x_1^k\}_{k\in\mathbb{N}}$.

Moreover, since the sequence $\{(x_1^k, \nu_k)\}_{k\in\mathbb{N}}$ satisfies Condition 1 in Definition \ref{d:pglds} and $X_1$ is compact,  by  \cite[Lemma 3]{cohen2025alternating}, we have $\nu_k\to0$ which implies $s\|y_1^k-\tilde{y}_1^N(x_1^k)\|\to0$. Note from Lemma \ref{l:properties} that $\tilde{y}_1^N(\cdot)$ is continuous, we have in part (i),
\[
\lim_{k\to\infty, k\in\mathcal{K}}\|y_1^k-\tilde{y}_1^N(\bar{x}_1^N)\| \leq \lim_{k\to\infty, k\in\mathcal{K}}(\|y_1^k-\tilde{y}_1^N(x^k_1)\|+\|\tilde{y}_1^N(x^k_1)-\tilde{y}_1^N(\bar{x}_1^N)\|)=0,
\]
the subsequence convergence of $\{y_1^k\}$; and in part (ii),
\[
\lim_{k\to\infty}\|y_1^k-\tilde{y}_1^N(\bar{x}_1^N)\| \leq \lim_{k\to\infty}(\|y_1^k-\tilde{y}_1^N(x^k_1)\|+\|\tilde{y}_1^N(x^k_1)-\tilde{y}_1^N(\bar{x}_1^N)\|)=0,
\]
the global convergence of $\{y_1^k\}$.
 \hfill $\square$
\end{proof}

\subsection{A semi-smooth Newton method for problem \eqref{eq:nonsmooth-equation}}


By Remark \ref{r:delta}, in step 3 of Algorithm 1, for $i=1, \cdots, N$ and every iteration $k$, our purpose is to find a $\pi$-component of an inexact solution of \eqref{eq:nonsmooth-equation}, denoted by $(\tilde{\pi}_{x_2^k}^{i}, \tilde{\pi}_{y_2^k}^{i})$, such that $\|\tilde{\pi}_{x_2^k}^{i} - \pi_{x_2^k} \|\leq \frac{\delta^k}{\bar{a}}$ and $\|\tilde{\pi}_{y_2^k}^{i} - \pi_{y_2^k}\|\leq \frac{\delta^k}{\bar{t}}$. In this section,  we apply a semi-smooth Newton method to find a $\frac{\delta_k}{\max\{\bar{a}, \bar{t}\}}$-solution of \eqref{eq:nonsmooth-equation} for each  $i=1, \cdots, N$ and $k\geq0$.


For  a fixed $(x_2,y_2,\xi)$, let
\[
M_1\in\begin{pmatrix}
\partial^2_{x_2x_2}F_2(x_2, y_2, \xi) & \partial^2_{x_2y_2}F_2(x_2, y_2, \xi) \\
-\partial^2_{y_2x_2}F_2(x_2, y_2, \xi) &  -\partial^2_{y_2y_2}F_2(x_2, y_2, \xi)
\end{pmatrix},\,\,
M_2:=\begin{pmatrix}
W(\xi)&0\\
 0 & B(\xi)
\end{pmatrix},
\]
$U\in \mathbb{R}^{(l_2+s_2)\times (l_2+s_2)}$ be a diagonal matrix with $u_{ii} \in [0,1]$, $i=1, \cdots, l_2+s_2$. Note that $M_2$ is of full row rank.
Then all matrices belonging to $\partial_\mu H(\cdot,\xi)$ at $(\mu,\xi)$ are of the form as follows:
\[
J H(\mu,\xi)=\begin{pmatrix}
M_1 & M_2^\top\\
(U-I)M_2& U
\end{pmatrix}.
\]
Note that $M_1$ and $M_2$ depend on $(x_2, y_2, \xi)$, but are abbreviated as $M_1$ and $M_2$ without explicit variables for notational simplicity.

\begin{proposition}\label{p:lowerbound}
Suppose that Assumptions \ref{a:second-stage}-\ref{a:xicompact-regular} hold. 
There exists $\underline{\lambda}>0$ such that for any given $(x_1, y_1, \xi)\in X_1\times Y_1\times \Xi$, every element $J H(\mu,\xi)\in \partial_\mu H(\mu,\xi)$ is nonsingular, and  $\|J H(\mu,\xi)^{-1}\|\leq \frac{1}{\sqrt{\underline{\lambda}}}$.
\end{proposition}

\begin{proof}
Since $F_2 (\cdot, \cdot, \xi)$ is $\sigma$-strongly convex-strongly concave for any fixed $\xi\in\Xi$,
we have $\det(M_1)>0$ for any $(x_2, y_2, \xi)\in \check{X}_2\times \check{Y}_2\times\Xi$. By the Schur complement, the determinant of $JH(\mu, \xi)$ is given by
\[
\det(JH(\mu,\xi))=\det(M_1)\det(U-(U-I)M_2M_1^{-1}M_2^\top).
\]
Furthermore, under Assumption \ref{a:second-stage} (iii), $M_2M_1^{-1}M_2^\top$ is positive definite. It then follows from   \cite{gabriel1997smoothing,chen2006computation} that the matrix $U-(U-I)M_2M_1^{-1} M_2^\top$ is nonsingular. Therefore, $JH(\mu,\xi)$ is nonsingular.

Moreover, under Assumptions \ref{a:second-stage} and \ref{a:xicompact-regular}, 
from the continuity of eigenvalues w.r.t. symmetric matrices, the compactness of $\check{X}_2\times \check{Y}_2$, the outer semicontinuity and local boundedness of $\partial_\mu H$, there exists $\underline{\lambda}$ such that
\[
0<\underline{\lambda}:=\min_{\substack{(x_2, y_2, \pi_{x_2}, \pi_{y_2})\in \check{X}_2\times \check{Y}_2\times\mathbb{R}^{2l}, \\ \xi\in\Xi, J H(\mu,\xi)\in\partial_\mu H(\mu,\xi)}}\lambda_{min}((J H(\mu,\xi))^\top J H(\mu,\xi))
\]
and then $\|J H(\mu,\xi)^{-1}\|\leq\frac{1}{\sqrt{\underline{\lambda}}}$ for any $\mu\in\check{X}_2\times \check{Y}_2\times\mathbb{R}^{2l}$ and $\xi\in\Xi$.
 \hfill $\square$
\end{proof}

\begin{definition}\cite{pang1990newton,qi1993convergence}
The nonlinear function $H(\cdot, \xi)$ is semi-smooth, if $H(\cdot, \xi)$ is locally Lipschitz and for all $d\in \mathbb{R}^{n_2+m_2+l_2+s_2}$ such that the following limit exists:
\[
\lim_{G\in \partial_\mu H(\mu+t\tilde{d}, \xi), \tilde{d}\to d, t\downarrow 0}G\tilde{d}.
\]
\end{definition}

\begin{assumption}\label{a:F2-semismooth}
For any $\xi\in\Xi$, $\nabla_{x_2}F_2(\cdot, \cdot, \xi)$ and $\nabla_{y_2}F_2(\cdot, \cdot, \xi)$ are semi-smooth over $\check{X}_2\times \check{Y}_2$.
\end{assumption}

\begin{proposition}\label{p:semismooth}
Under Assumption \ref{a:F2-semismooth}, for any $\xi\in\Xi$, $H(\cdot, \xi)$ is semi-smooth over $\check{X}_2\times \check{Y}_2\times \mathbb{R}^{l_2+s_2}$.
\end{proposition}

\begin{proof}
Since  the ``min" operator is a semi-smooth operator \cite{pang1990newton,qi1993convergence}, and the corresponding functions in ``min" operator are linear functions, $\min( \cdot, h(\xi)-T(\xi)x_1-W(\xi)\cdot)$ and
$\min( \cdot, c(\xi)-A(\xi)y_1-B(\xi)\cdot)$ are semi-smooth. Combining with Assumption \ref{a:F2-semismooth},  $H(\cdot, \xi)$ is semi-smooth over $\check{X}_2\times \check{Y}_2\times \mathbb{R}^{l_2+s_2}$.
 \hfill $\square$
\end{proof}

Then we apply semi-smooth Newton method \cite{pang1990newton,qi1993convergence} to solve the system of nonlinear equations \eqref{eq:nonsmooth-equation}.
\begin{algorithm}[H]
		\caption{Semi-smooth Newton method for nonlinear equation \eqref{eq:nonsmooth-equation} }
		\begin{algorithmic}[1]
			\Require initial point $\mu_0$ and $t:=0$
               \While{$\|H(\mu^t, \xi)\|> \epsilon$}
.              \State  Solve $G(\mu^t, \xi)d^t=-H(\mu^t, \xi)$ for $d^t$, where $G(\mu^t, \xi)\in J H(\mu^t,\xi)$;
               \State   $\mu^{t+1}=\mu^t+\alpha^td^t$, $t=t+1$,
               \EndWhile
              \Ensure $\mu^{t+1}$
		\end{algorithmic}
	\end{algorithm}

We show the convergence analysis of Algorithm 2 based on \cite[Corollary 3.4 and Theorem 4.3]{qi1993convergence} as follows.

\begin{theorem}\label{t:convergence2}
Suppose that Assumptions \ref{a:second-stage}, \ref{a:xicompact-regular} 
and \ref{a:F2-semismooth} hold,
$\beta\in(0,1)$, $\sigma\in(0, 1/2)$, $\alpha^t=\beta^{m_t}$ and $m_t$ is  the first nonnegative integer $m$ such that $H(\mu^t,\xi)-H(\mu^t+\beta^{m}d^t,\xi)\geq -\sigma\beta^{m}H'(\mu^t, \xi; d^t)$, where
\[
H'(\mu, \xi; d):=\max_{J H(\mu,\xi)\in \partial_\mu H(\mu,\xi)}J H(\mu,\xi)^\top d.
\]
 Then for any $\xi\in\Xi$,
\begin{itemize}
\item[(i)] $\{\|H(\mu^t,\xi)\|\}$ converges to $0$ superlinearly and $\{\alpha_t\}$ eventually becomes $1$;
\item[(ii)] the entire sequence $\{\mu^t\}$ satisfies $\|\mu^t-\mu^*\|\leq \frac{1}{\sqrt{\underline{\lambda}}}\|H(\mu^t, \xi)\|$, where $\mu^*$ is the solution of problem \eqref{eq:nonsmooth-equation}.
\end{itemize}
\end{theorem}
\begin{proof}
Under Assumptions \ref{a:second-stage}-\ref{a:xicompact-regular}, by Lemma \ref{l:2stage-bound-continuity}, for any $\xi\in\Xi$, there exists a unique $\mu^*\in\mathbb{R}^{n_2+m_2+l_2+s_2}$ such that $H(\mu^*, \xi)=0$.

Under  Assumption 
\ref{a:F2-semismooth}, by
Propositions \ref{p:lowerbound}-\ref{p:semismooth},  for any $\xi\in\Xi$, $H(\cdot, \xi)$ is semi-smooth and strongly B-D regular at $\mu^*$, that is, for any $d\neq0$,
$
H'(\mu^*, \xi; d)\neq0.
$
Then by \cite[Corollary 3.4 and Theorem 4.3]{qi1993convergence}, we have $\|H(\mu^t,\xi)\|$ converges to $0$ superlinearly, and $\alpha_k$ eventually becomes $1$. Moreover, by \cite[Proposition 2.6.5]{clarke1990optimization}, there exists $\bar{\mu}^t$ in the line segment $[\mu^t, \mu^*]$ such that
\[
H(\mu^t,\xi) = JH(\bar{\mu}^t, \xi)(\mu^t-\mu^*),
\]
then $\|\mu^t-\mu^*\|\leq \|JH(\bar{\mu}^t,\xi)^{-1}\|\|H(\mu^t,\xi)\|\leq \frac{1}{\sqrt{\underline{\lambda}}}\|H(\mu^t,\xi)\|$.
 \hfill $\square$
\end{proof}

With Algorithm 2 and Theorem \ref{t:convergence2}, we can replace steps 2-4 in Algorithm 1 with the following procedure:
\begin{algorithmic}[1]
\renewcommand{\alglinenumber}[1]{} 
\For{$i = 1, 2, \cdots, N$}
    \State solve \eqref{eq:nonsmooth-equation} with given $(x_{1}^{k}, y_{1}^{k})$ and $\xi^{i}$ by Algorithm 2 (semi-smooth Newton method) with $\| H(\mu^{t+1}, \xi^i) \| < \epsilon^{k}$, where $\epsilon^{k} := \frac{\delta^{k} \sqrt{\underline{\lambda}}}{\max\{\bar{a}, \bar{t}\}}$, and obtain
    $\mu^{k,i} = \mu^{t+1}$.
\EndFor
\end{algorithmic}
The convergence result (Theorem~\ref{t:mainconvergence}) still holds for Algorithm 1.

\subsection{Numerical experiments}

In this section, we present preliminary numerical experiments about a two-stage stochastic two-player zero-sum game to validate both our theoretical framework and proposed algorithm. All the numerical experiments are conducted using MATLAB 2022b on a PC with 12th Gen Intel(R) Core(TM) i7-12700 running at 2.10 GHz and 32 GB of RAM.

{\bf Two-player two-stage stochastic zero-sum game:}  Two-stage quadratic SNEPs are investigated under the monotone condition \cite{pang2017two,zhang2019two,lei2020synchronous}. Here we consider a nonconvex-nonsmooth extension of the two-stage SNEPs,
as the two-stage stochastic minimax problem \eqref{eq:ts-minimax}-\eqref{eq:st-minimax}, where
\[
F_1(x_1, y_1):=\|x_1\|_1 - \frac{1}{2}x_{1}^\top Q_{1}x_{1} + (d_1)^\top x_{1}  +  x_{1}^\top O_{1}y_{1} - \frac{1}{2}y_{1}^\top S_{1}y_{1} - (t_1)^\top y_{1},
\]
\[
F_2(x_2, y_2, \xi):= \frac{1}{2}x_{2}^\top Q_{2}(\xi)x_{2} + (d_2(\xi))^\top x_{2} +x_{2}^\top O_{2}(\xi)y_{2} - \frac{1}{2}y_{2}^\top S_{2}(\xi)y_{2} - (t_2(\xi))^\top y_{2},
\]
$\xi:\Omega\to \Xi\subset\mathbb{R}^l$ is a random variable, $X_1:=[{\rm lb},{\rm ub}]^{n_1}\subset\mathbb{R}^{n_1}$, $Y_1=\mathbb{R}^{m_1}$, $Q_1\in\mathbb{R}^{n_{1}\times n_{1}}$, $O_1\in\mathbb{R}^{n_1\times m_1}$, $S_1\in\mathbb{R}^{m_{1}\times m_{1}}$, $d_1\in\mathbb{R}^{n_1}$,  $t_1\in\mathbb{R}^{m_1}$,
$$
X_2(x_1, \xi):=\{x_2\in \mathbb{R}^{n_2} : T(\xi)x_1+W(\xi)x_2\leq h(\xi) \},
$$
$$
Y_2(y_1, \xi):=\{y_2\in  \mathbb{R}^{m_2} : A(\xi)y_1+B(\xi)y_2\leq c(\xi) \},
$$
  $Q_2(\xi)\in\mathbb{R}^{n_{2}\times n_{2}}$, $S_2(\xi)\in\mathbb{R}^{m_{2}\times m_{2}}$  are symmetric positive definite matrices,  $O_2(\xi)\in\mathbb{R}^{n_2\times m_2}$,   $d_2(\xi)\in \mathbb{R}^{n_2}$,  $t_2(\xi)\in \mathbb{R}^{m_2}$, $T(\xi)\in\mathbb{R}^{l_2\times n_1}$, $W(\xi)\in\mathbb{R}^{l_2\times n_2}$, $A(\xi)\in\mathbb{R}^{s_2\times m_1}$, $B(\xi)\in\mathbb{R}^{s_2\times m_2}$, $h(\xi)\in\mathbb{R}^{l_2}, c(\xi)\in\mathbb{R}^{s_2}$ for  a.e. $\xi\in\Xi$.

The setting of this two-stage stochastic minimax problem
is as follows.

\begin{itemize}
\item
First stage setting: $n_1=3, m_1=2$, $Q_1=0.1 I_{n_1}$, $S_1=I_{m_1}$, $O_1\in\mathbb{R}^{n_1\times m_1}$, $d_1\in \mathbb{R}^{n_1}$,  $t_1\in \mathbb{R}^{m_1}$ with all elements randomly generated from a uniform distribution over $[0,1]$.

\item Fixed matrices and vectors in second-stage problems: $n_2=4, m_2=3, l_2=s_2=2$, $\bar{Q}_2={\rm diag}(1,2,3,4)$,  $\bar{S}_{2}=I_{m_2}$, 
$\bar{h}=(0.1, 0.1)^\top$, $\bar{c}=(0.1, 0.1)^\top$. $\bar{O}_2$, $\bar{T}$, $\bar{A}$, $\bar{d}_2$, $\bar{t}_2$ are randomly generated matrices or vectors with corresponding dimensions, where each element is drawn from a uniform distribution over $[0,1]$.

\item Random matrices and vectors in second-stage problems: A $49$-dimensional random vector $\xi$, where each component is independently sampled from a uniform distribution over $[-1,1]$. The first $\frac{n_2\times(n_2+1)}{2}$ elements form the upper triangular part of the random symmetric matrix $\tilde{Q}_2(\xi)\in\mathbb{R}^{n_2\times n_2}$; the next $\frac{m_2\times(m_2+1)}{2}$ elements form the upper triangular part of the random symmetric matrix $\tilde{S}_2(\xi)\in\mathbb{R}^{m_2\times m_2}$; the remaining elements construct random matrices and vectors $\tilde{T}(\xi)$, $\tilde{A}(\xi)$, $\tilde{d}_2(\xi)$, $\tilde{t}_2(\xi)$. For each experimental setup, we generate $N$ i.i.d. samples.
\item Second-stage setting: $Q_2(\xi)=\tau\bar{Q}_2+0.1\tilde{Q}_2(\xi)$, $S_2(\xi)=\tau\bar{S}_2+0.1\tilde{S}_2(\xi)$, $O_2(\xi)=\bar{O}_2+0.1\tilde{O}_2(\xi)$, $W(\xi)=(I_{l_2}, {\bm 0})$, $B(\xi)=(I_{s_2}, {\bm 0})$, $h(\xi)=\bar{h}+0.1\tilde{h}(\xi)$, $c(\xi)=\bar{c}+0.1\tilde{c}(\xi)$, $T(\xi)=\bar{T}+0.1\tilde{T}(\xi)$,  $d_2(\xi)=\bar{d}_2+0.1\tilde{d}_2(\xi)$, $t_2(\xi)=\bar{t}_2+0.1\tilde{t}_2(\xi)$, where ${\bm 0}$ denotes the zero matrix of appropriate dimension as specified by the context.

\item The residual value (Res.val) is defined as follows:
\[
\begin{array}{lll}
{\rm Res.val}(x_1^k, y_1^k)&:=\|-S_1y_1^k-t_1+O_1^\top x_1^k+\tilde{v}_y^k\|\\
     & +  \displaystyle{\min_{\eta\in\partial \|x_1^k\|_1}}\|x_1^k-{\rm mid}(x_1^k-\eta-w^k, {\rm ub, lb})\|,
\end{array}
\]
where $w^k=(d_1 + O_1y_1^k + \tilde{v}_x^k-Q_1x_1^k)$ and ``mid''
denotes the component-wise median operator on three vectors. In our computation, we use a minimizer $\eta^k$ of the above minimization problem  in the following form:
$$
(\eta^k)_i =
\begin{cases}
-1, & \text{if } (x^k_1)_i <0  \,\,  {\rm or} \, \,(x^k_1)_i=0 \,\,  {\rm and} \,\,   w^k_i > 1, \\
1, & \text{if } (x^k_1)_i >0 \,  \, {\rm or} \, \,(x^k_1)_i=0 \,\,  {\rm and} \,\,   w^k_i < 1, \\
-w^k_i, & \text{if } (x^k_1)_i=0 \, \, {\rm and} \,\,   w^k_i \in [-1,1].
\end{cases}
$$

\item
The termination criterion is set as Res.val$\leq 10^{-4}$.

\end{itemize}

{\bf  Experiment 1:} We set the sample size to $N=50$, with lower bound ${\rm lb}=-10$ and upper bound ${\rm ub}=10$.
 Five initial points $(x_1^0, y_1^0)$ are randomly generated with $x_1^0$ components uniform on $[7,10]$ and $y_1^0$ components uniform on $[0,1]$. For $\tau=0.1$ and $0.5$,
 we randomly generated five problems with five different initial points per $\tau$-value. The convergence behavior of IPPGDA algorithm regarding  the residual values is shown in Fig. \ref{fig:total}.

 Note that $Q_2(\xi)=\tau\bar{Q}_2+0.1\tilde{Q}_2(\xi)$, $S_2(\xi)=\tau\bar{S}_2+0.1\tilde{S}_2(\xi)$.
Fig. \ref{fig:total} demonstrates that smaller $\tau$ values lead to more divergent convergence paths, attributable to stronger stochastic influences at lower $\tau$ values.
Moreover,  when $\tau<0.1$, $Q_2(\xi)$ and $S_2(\xi)$ may lose positive definiteness, leading to potential algorithmic non-convergence.
\begin{figure}[H]
    \centering
    \subfloat{
        \includegraphics[width=0.48\textwidth, height=5cm]{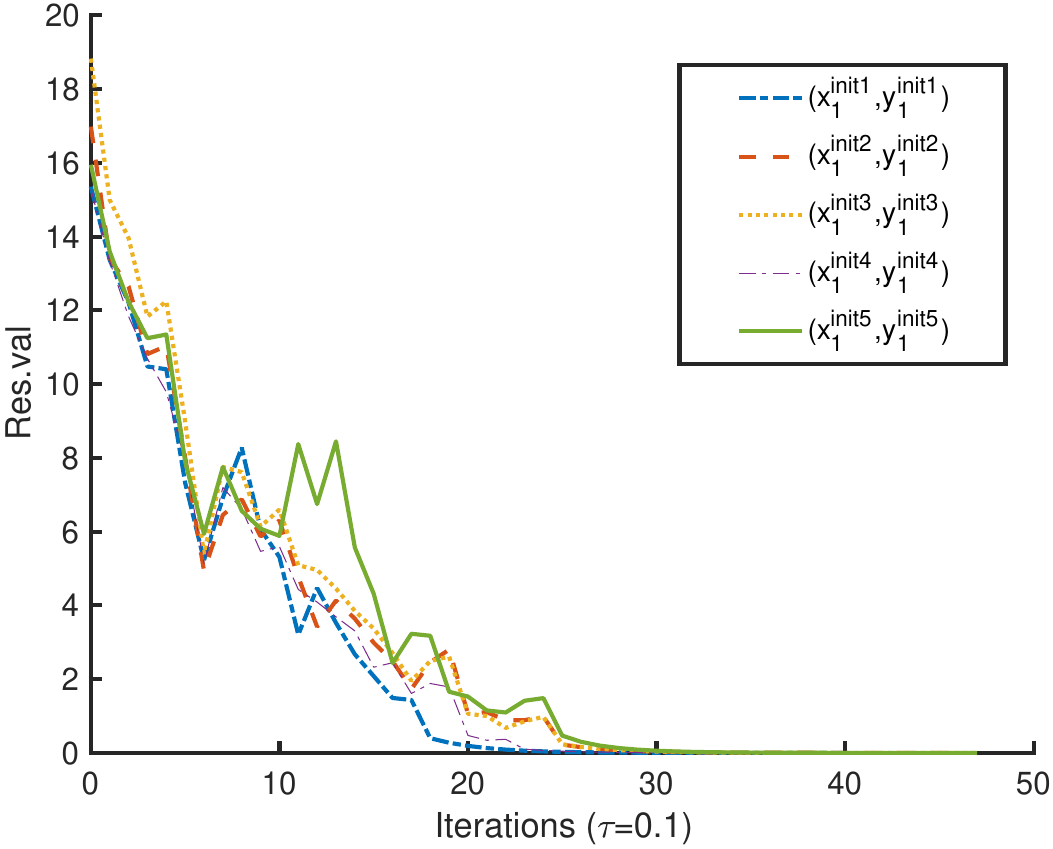}
    }
    \hfill
    \subfloat{
        \includegraphics[width=0.48\textwidth, height=5cm]{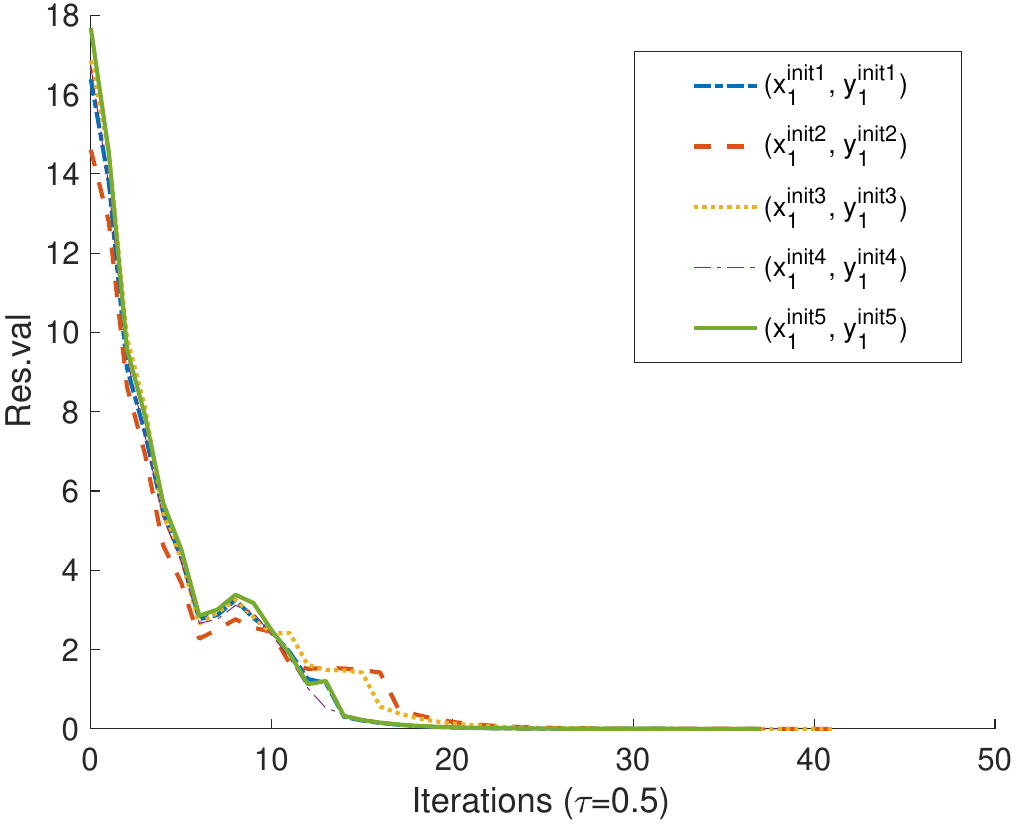}
    }
    \caption{Residual values versus iteration numbers with different starting points and random matrices for $\tau=0.1$ and $0.5$, respectively}
   \label{fig:total}
\end{figure}

{\bf  Experiment 2:} We conduct numerical experiments with sample sizes $N=10,50,200,500,1000, 3000$. The feasible set  $X_1$ is defined by the box constraints $ [-10, 10]^{3}$ and $[-20, 20]^{3}$.
For each combination of sample sizes and constraint sets, we randomly generate 30 instances of the two-stage stochastic minimax problem and solve them using IPPGDA algorithm. The resulting minimax values are shown in Fig. \ref{fig:fixed}.

\begin{figure}[H]
    \centering
    \subfloat{
        \includegraphics[width=0.48\textwidth, height=5cm]{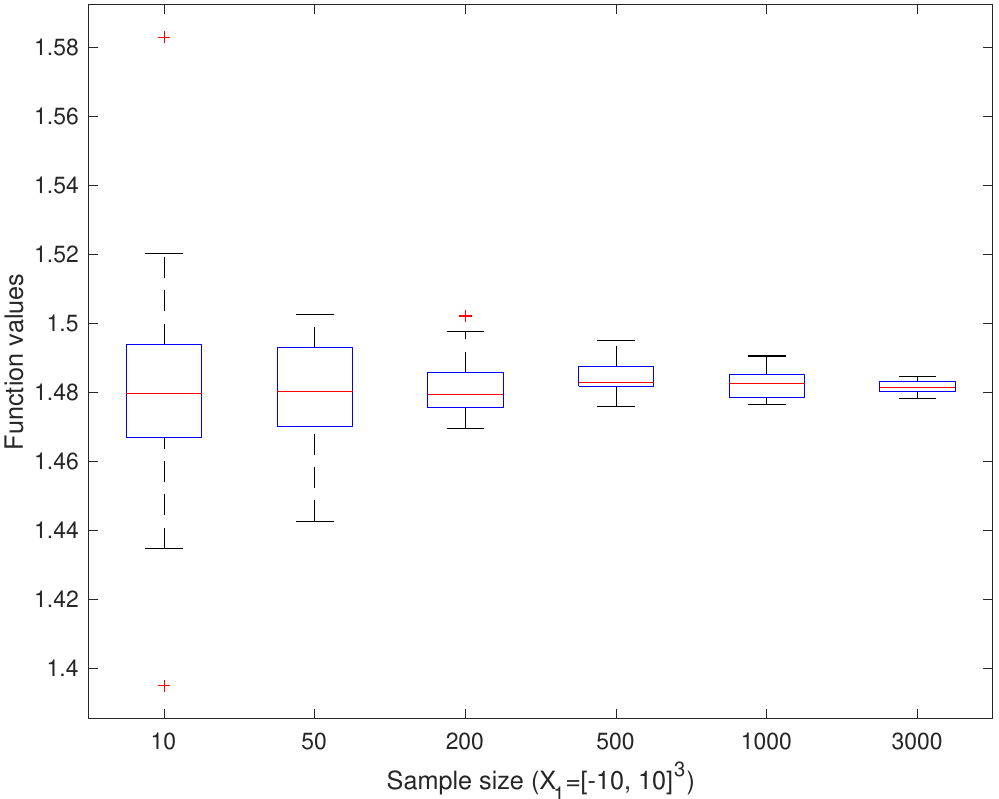}
    }
    \hfill
    \subfloat{
        \includegraphics[width=0.48\textwidth, height=5cm]{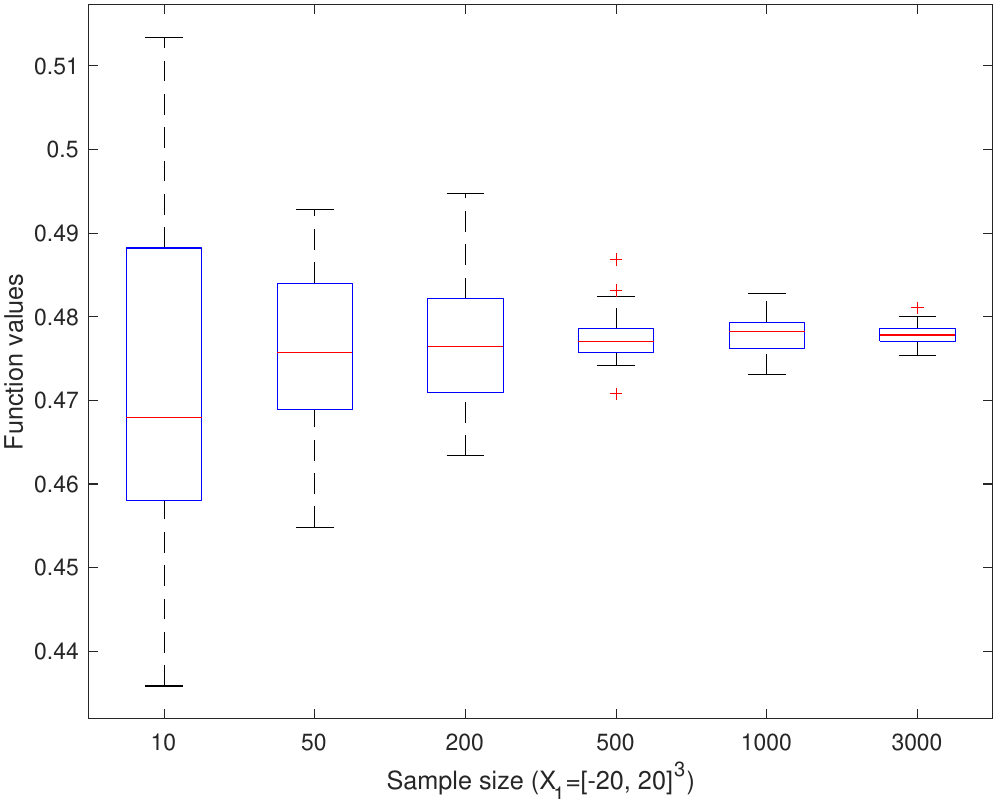}
    }
    \caption{Convergence of the SAA problems when $X_1 = [-10, 10]^{3}$ and $X_1= [-20, 20]^{3}$, respectively}
   \label{fig:fixed}
\end{figure}

Fig. \ref{fig:fixed} shows that the SAA problems exhibit convergent behavior as the sample size increases, which is consistent with theoretical expectations. Moreover, a larger feasible set $X_1$ leads to a smaller objective function value, which also aligns with our expectations.

 \section{Conclusion}\label{sec13}

In this paper, we introduce the two-stage stochastic minimax problem (\ref{eq:ts-minimax})-(\ref{eq:st-minimax}), analyze the Lipschitz continuity of second-stage minimax value function and solution functions, along with the properties and relationships among saddle points, minimax points and KKT points. We further prove the convergence of the SAA method for problem (\ref{eq:ts-minimax})-(\ref{eq:st-minimax}), and discuss exponential convergence rates as the sample size goes to infinity.
To solve the SAA problem, we propose an IPPGDA algorithm. The algorithm utilizes a semi-smooth Newton approach to solve second-stage subproblems, obtaining approximate gradients of the second-stage minimax value function which are subsequently integrated into an inexact first-stage proximal gradient scheme for the minimax problem. Preliminary numerical experiments demonstrate the effectiveness of IPPGDA algorithm while validating the convergence properties of the SAA approach.


\section{Appendix: A generalization of Theorem 2G.8 in \cite{dontchev2009implicit} (implicit function theorem for stationary points)}

In \cite{dontchev2009implicit}, Dontchev and Rockafellar gave Theorem 2G.8,  an  implicit function theorem for stationary points for
the parametric nonlinear programming problem in the form
\begin{equation}\label{eq:imp-model}
    \begin{array}{cl}
       \min_z  & h_0(p,z)  \\
        {\rm s.t.} & h_i(p,z)\leq 0, \; i\in[1,s],\\
             & h_i(p,z)=0, \; i\in [s+1, m]
    \end{array}
\end{equation}
with parameter $p$, where $h_i:\mathbb{R}^k\times \mathbb{R}^n \to \mathbb{R}$ is twice continuously differentiable, $i=0, \cdots, m$.  
Let
\[
L(p, z, \tau)=h_0(p,z) + \tau_1h_1(p,z) + \cdots + \tau_mh_m(p,z).
\]
For a fixed $p$, the variational inequality capturing the associated first-order conditions of problem \eqref{eq:imp-model} is
\begin{equation}\label{eq:KKT}
(0,0) \in H(p,z,\tau) + N_E(z,\tau),
\end{equation}
where
\[
H(p, z,\tau) = (\nabla_z L(p,z,\tau)^{\top}, -\nabla_\tau L(p, z,\tau)^{\top})^{\top}, \; E=\mathbb{R}^n\times \mathbb{R}_+^s\times \mathbb{R}^{m-s}.
\]
Let
\begin{equation}\label{eq:solution-mapping}
S(p):= \{(z, \tau)| (0,0)\in H(p, z,\tau) + N_E(z,\tau)\}
\end{equation}
 be the solution mapping of the generalized equation \eqref{eq:KKT} and assume that $S(p)$ is nonempty in a neighborhood of $\bar{p}$.

To study the two-stage stochastic minimax problem \eqref{eq:ts-minimax}-\eqref{eq:st-minimax},
 we need an implicit function theorem for problems \eqref{eq:f21} and \eqref{eq:bbfmin}, where the objective functions $F_2$ and  $f_{21}$ are continuously differentiable but not twice continuously differentiable. Now we give implicit function theorem for  \eqref{eq:imp-model} where
 $h_i$ is  continuously differentiable and $\nabla_zh_i$ is Lipschitz continuous, $i=0, \cdots, m$.
The theorem is extended from \cite[Theorem 2G.8]{dontchev2009implicit}, and weakens the twice continuous differentiability of $h_i$ to the continuous differentiability of $h_i$.

 For any $(\bar{z}, \bar{\tau})\in S(\bar{p})$, we give an auxiliary problem of \eqref{eq:imp-model}. Let $\bar{G}\in \partial_{zz}^2 L(\bar{p}, \bar{z}, \bar{\tau})$,
\[
    \bar{h}_0(w, \bar{G}) := L(\bar{p}, \bar{z}, \bar{\tau}) + \langle \nabla_z L(\bar{p}, \bar{z}, \bar{\tau}), w \rangle + \frac{1}{2}\langle w, \bar{G}w \rangle
    \]
    and
    \[
    \bar{h}_i(w) := h_i(\bar{p}, \bar{z}) + \langle \nabla_z h_i(\bar{p}, \bar{z}), w \rangle, \; \forall i=1, \cdots, m.
\]
Then the auxiliary problem of \eqref{eq:imp-model} with parameters $\nu$ and $u_1, \cdots, u_m$ is
\begin{equation*}\label{eq:imp-auxiliary}
\begin{array}{lll}
\min_{w} & \bar{h}_0(w, \bar{G}) - \langle \nu, w  \rangle \\
{\rm s.t.} & \bar{h}_{i}(w) + u_i\leq 0, \,\,  i\in I_0,\\
& \bar{h}_{i}(w) + u_i=0, \,\,  i\in I/I_0,\\
& \bar{h}_{i}(w) + u_i {\;free\;}, \,\, i\in I_1,
\end{array}
\end{equation*}
where
$$
\begin{aligned}
&I = \{ i \in [1,m] : h_i(\bar{p},\bar{z}) = 0 \} \supset \{ s+1,\ldots,m \}, \\
&I_0 = \{ i \in [1,s] : h_i(\bar{p},\bar{z}) = 0 \ \text{and}\ \bar{\tau}_i = 0 \}, \\
&I_1 = \{ i \in [1,s] : h_i(\bar{p},\bar{z}) < 0 \}.
\end{aligned}
$$
For any $\bar{G}\in \partial_{zz}^2 L(\bar{p}, \bar{z}, \bar{\tau})$, let
\[
\bar{L}(w, \varsigma, \bar{G})-\langle \nu, w\rangle + \langle \varsigma, u \rangle:= \bar{h}_0(w, \bar{G}) - \langle \nu, w\rangle + \varsigma_1(\bar{h}_1(w) + u_1) + \cdots +\varsigma_m(\bar{h}_m(w) + u_m).
\]
The corresponding first-order conditions are given by the variational inequality
\[
(0,0)\in (\nabla_w \bar{L}(w, \varsigma, \bar{G}), -\nabla_\varsigma \bar{L}(w, \varsigma, \bar{G})) - ( \nu, u) + N_{\bar{E}}(w, \varsigma),
\]
where $\bar{E} = \mathbb{R}^n\times \{\varsigma\in \mathbb{R}^m: \varsigma_i\geq 0, \; \forall i\in I_0 \; {\rm and}\; \varsigma_i= 0, \; \forall i\in I_1\}$. Let
\begin{equation}\label{eq:solution-aux}
\begin{array}{l}
M^+=\{w\in\mathbb{R}^n: w\bot\nabla_zh_i(\bar{p}, \bar{z}) \inmat{ for all } i\in I \setminus I_0\},\\
M^-=\{w\in\mathbb{R}^n: w\bot\nabla_zh_i(\bar{p}, \bar{z}) \inmat{ for all } i\in I\},
\end{array}
\end{equation}
$\bar{S}(\nu,u, \bar{G}):=\{ (w, \varsigma):  (0,0)\in (\nabla_w \bar{L}(w, \varsigma, \bar{G}), -\nabla_\varsigma \bar{L}(w, \varsigma, \bar{G})) - \langle \nu, u\rangle + N_{\bar{E}}(w, \varsigma)\}.$

Now we present the implicit function theorem for stationary points of \eqref{eq:imp-model}.

\begin{theorem}\label{t:imp}
Let \((\bar{z}, \bar{\tau}) \in S(\bar{p})\) for the mapping \(S\) in \eqref{eq:solution-mapping}, constructed from functions \(h_i\) that are continuously differentiable such that $\nabla_zh_i$ is Lipschitz continuous.
Assume
$$\begin{array}{l}
\text{(A1): For any } \bar{G}\in \partial_{zz}^2 L(\bar{p}, \bar{z}, \bar{\tau}), \bar{S}(\cdot, \cdot, \bar{G})
\text{ has a Lipschitz continuous }\\
\text{single-valued localization } \bar{s} \text{ around } (0,0) \text{ for } (0,0).
\end{array}
$$
Then $S$ has a Lipschitz continuous single-valued localization $s$ around $\bar{p}$ for $(\bar{z}, \bar{\tau})$.

Moreover,  condition (A1)
is necessary for the existence of a Lipschitz continuous single-valued localization of \(S\) around \(\bar{p}\) for \((\bar{z}, \bar{\tau})\). 

In particular, $\bar{S}(\cdot, \cdot, \bar{G})$ is sure to have the property in (A1)   when the following conditions are both fulfilled:
\begin{itemize}
\item[(a)]  the gradients  $\nabla_x h_i(\bar{p}, \bar{z})$  for  $i \in I$  are linearly independent;
\item[(b)] there exists $\iota>0$ such that for all $\bar{G}\in \partial_{zz}^2 L(\bar{p}, \bar{z}, \bar{\tau})$, $\langle w, \bar{G} w \rangle > \iota$  for every nonzero $ w \in M^+$  with  $\bar{G} w \perp M^-,$
with \(M^+\) and \(M^-\) as in \eqref{eq:solution-aux}.
\end{itemize}
On the other hand, condition (A1) always entails at least (a).
\end{theorem}

\begin{proof}
  The result is obtained by applying \cite[Theorem 3]{izmailov2014strongly} with \cite[Theorem 2E.6]{dontchev2009implicit} to the variational inequality \eqref{eq:KKT}.

Let $h=(h_1, \cdots, h_m)$.
  Then
  $\nabla_\tau L(p, z, \tau) =h(p, z)$, and
  the Clarke generalized Jacobian of $H$ at $(\bar{p}, \bar{z}, \bar{\tau})$ is
  \[
  JH(\bar{p}, \bar{z}, \bar{\tau})=
  \begin{pmatrix}
      \partial_{zz}^2 L(\bar{p}, \bar{z}, \bar{\tau}) & \nabla_zh(\bar{p}, \bar{z})\\
      -\nabla_zh(\bar{p}, \bar{z}) & 0
  \end{pmatrix}.
  \]
  For any $\bar{G}\in \partial_{zz}^2 L(\bar{p}, \bar{z}, \bar{\tau})$, let
  \[
  J\bar{H}(\bar{p}, \bar{z}, \bar{\tau}, \bar{G}):=
  \begin{pmatrix}
      \bar{G} & \nabla_zh(\bar{p}, \bar{z})\\
      -\nabla_zh(\bar{p}, \bar{z}) & 0
  \end{pmatrix}.
  \]

Note that the critical cone to the polyhedral convex cone set $E$ is
  \[
  K_E(\bar{z}, \bar{\tau}, -H(\bar{p}, \bar{z}, \bar{\tau}))=\bar{E},
  \]
  (A1)  is equivalent to, for any $\bar{G}\in \partial_{zz}^2 L(\bar{p}, \bar{z}, \bar{\tau})$,
  \[
  \delta\in H(\bar{p}, \bar{z}, \bar{\tau}) + J\bar{H}(\bar{p}, \bar{z}, \bar{\tau}, \bar{G})((z^\top, \tau^\top)^\top)+N_E((z^\top, \tau^\top)^\top)
  \]
  has a Lipschitz continuous single-valued localization around $0$ for $(\bar{z}, \bar{\tau})$, that implies the variational inequality \eqref{eq:KKT} is parametrically CD-regular (see \cite[Definition 4]{izmailov2014strongly}). Note
  also that, by the Lipschitz continuity of $\nabla_zh_i$, we have $H$ is Lipschitz continuous. Then by \cite[Theorem 4]{izmailov2014strongly}, $S$ has a Lipschitz continuous single-valued localization $s$ around $\bar{p}$ for $(\bar{z}, \bar{\tau})$.

  The rest of the proof is to establish that, for any $\bar{G}\in \partial_{zz}^2 L(\bar{p}, \bar{z}, \bar{\tau})$, (A1)  implies (a) and prove (a) and (b) are the sufficient condition of (A1). They are the same as the corresponding proof in \cite[Theorem 2G.8]{dontchev2009implicit}, and we omit them here.
 \hfill $\square$
\end{proof}

\bibliographystyle{spmpsci}      
\bibliography{ref}

@book{shapiro2021lectures,
  title={Lectures on Stochastic Programming: Modeling and Theory},
  author={Shapiro, Alexander and Dentcheva, Darinka and Ruszczynski, Andrzej},
  year={2021},
  publisher={SIAM}
}

@book{birge1997introduction,
  title={Introduction to Stochastic Programming},
  author={Birge, John R and Louveaux, Francois},
  year={1997},
  publisher={Springer}
}

@article{liu2009two,
  title={A two-stage stochastic programming model for transportation network protection},
  author={Liu, Changzheng and Fan, Yueyue and Ord{\'o}{\~n}ez, Fernando},
  journal={Comput. Oper. Res.},
  volume={36},
  pages={1582--1590},
  year={2009},
  publisher={Elsevier}
}

@book{rockafellar1998variational,
  title={Variational Analysis},
  author={Rockafellar, R Tyrrell and Wets, Roger JB},
  year={1998},
  publisher={Springer}
}

@article{liu2021two,
  title={Two-stage stochastic optimization via primal-dual decomposition and deep unrolling},
  author={Liu, An and Yang, Rui and Quek, Tony QS and Zhao, Min-Jian},
  journal={IEEE Trans. Signal Process.},
  volume={69},
  pages={3000--3015},
  year={2021},
  publisher={IEEE}
}

@article{lee2018resource,
  title={Resource allocation for multi-channel underlay cognitive radio network based on deep neural network},
  author={Lee, Woongsup},
  journal={IEEE Commun. Lett.},
  volume={22},
  pages={1942--1945},
  year={2018},
  publisher={IEEE}
}

@article{v1928theorie,
  title={Zur theorie der gesellschaftsspiele},
  author={v. Neumann, J},
  journal={Math. Ann.},
  volume={100},
  pages={295--320},
  year={1928},
  publisher={Springer}
}

@article{chen2017two,
  title={Two-stage stochastic variational inequalities: an {ERM}-solution procedure},
  author={Chen, Xiaojun and Pong, Ting Kei and Wets, Roger JB},
  journal={Math. Program.},
  volume={165},
  pages={71--111},
  year={2017},
  publisher={Springer}
}

@article{rockafellar2017stochastic,
  title={Stochastic variational inequalities: single-stage to multistage},
  author={Rockafellar, R Tyrrell and Wets, Roger JB},
  journal={Math. Program.},
  volume={165},
  pages={331--360},
  year={2017},
  publisher={Springer}
}

@article{shapiro2002minimax,
  title={Minimax analysis of stochastic problems},
  author={Shapiro, Alexander and Kleywegt, Anton},
  journal={Optim. Methods Softw.},
  volume={17},
  pages={523--542},
  year={2002},
  publisher={Taylor \& Francis}
}

@article{lan2023novel,
  title={A Novel Catalyst Scheme for Stochastic Minimax Optimization},
  author={Lan, Guanghui and Li, Yan},
  journal={arXiv preprint arXiv:2311.02814},
  year={2023}
}

@article{chen2024near,
  title={Near-optimal algorithms for making the gradient small in stochastic minimax optimization},
  author={Chen, Lesi and Luo, Luo},
  journal={J. Mach. Learn. Res.},
  volume={25},
  pages={1--44},
  year={2024}
}

@article{xu2023unified,
  title={A unified single-loop alternating gradient projection algorithm for nonconvex-concave and convex-nonconcave minimax problems},
  author={Xu, Zi and Zhang, Huiling and Xu, Yang and Lan, Guanghui},
  journal={Math. Program.},
  volume={201},
  pages={635--706},
  year={2023},
  publisher={Springer}
}

@article{lei2020synchronous,
  title={On synchronous, asynchronous, and randomized best-response schemes for stochastic {N}ash games},
  author={Lei, Jinlong and Shanbhag, Uday V and Pang, Jong-Shi and Sen, Suvrajeet},
  journal={Math. Oper. Res.},
  volume={45},
  pages={157--190},
  year={2020},
  publisher={INFORMS}
}

@article{zhang2019two,
  title={Two-stage quadratic games under uncertainty and their solution by progressive hedging algorithms},
  author={Zhang, Min and Sun, Jie and Xu, Honglei},
  journal={SIAM J. Optim.},
  volume={29},
  pages={1799--1818},
  year={2019},
  publisher={SIAM}
}

@article{pang2017two,
  title={Two-stage non-cooperative games with risk-averse players},
  author={Pang, Jong-Shi and Sen, Suvrajeet and Shanbhag, Uday V},
  journal={Math. Program.},
  volume={165},
  pages={235--290},
  year={2017},
  publisher={Springer}
}

@article{beale1955minimizing,
  title={On minimizing a convex function subject to linear inequalities},
  author={Beale, Evelyn ML},
  journal={J. R. Stat. Soc. Ser. B. Stat. Methodol.},
  volume={17},
  pages={173--184},
  year={1955},
  publisher={Oxford University Press}
}

@article{dantzig1955linear,
  title={Linear programming under uncertainty},
  author={Dantzig, George B},
  journal={Manag. Sci.},
  volume={1},
  pages={197--206},
  year={1955},
  publisher={Informs}
}

@book{washburn2014two,
  title={Two-person Zero-sum Games},
  author={Washburn, Alan R},
  year={2014},
  publisher={Springer}
}

@book{clarke1990optimization,
  title={Optimization and Nonsmooth Analysis},
  author={Clarke, Frank H},
  year={1990},
  publisher={SIAM}
}

@article{gabriel1997smoothing,
  title={Smoothing of mixed complementarity problems},
  author={Gabriel, Steven A and Mor{\'e}, Jorge J},
  journal={Complementarity and Variational Problems: State of the Art},
  volume={92},
  pages={105--116},
  year={1997},
  publisher={SIAM Philadelphia}
}

@article{chen2006computation,
  title={Computation of error bounds for {P}-matrix linear complementarity problems},
  author={Chen, Xiaojun and Xiang, Shuhuang},
  journal={Math. Program.},
  volume={106},
  pages={513--525},
  year={2006},
  publisher={Springer}
}

@article{qi1993convergence,
  title={Convergence analysis of some algorithms for solving nonsmooth equations},
  author={Qi, Liqun},
  journal={Math. Oper. Res.},
  volume={18},
  pages={227--244},
  year={1993},
  publisher={INFORMS}
}

@article{pang1990newton,
  title={Newton's method for {B}-differentiable equations},
  author={Pang, Jong-Shi},
  journal={Math. Oper. Res.},
  volume={15},
  pages={311--341},
  year={1990},
  publisher={INFORMS}
}

@book{facchinei2003finite,
  title={Finite-dimensional Variational Inequalities and Complementarity Problems},
  author={Facchinei, Francisco and Pang, Jong-Shi},
  year={2003},
  publisher={Springer}
}

@article{bian2024nonsmooth,
  title={Nonsmooth convex--concave saddle point problems with cardinality penalties},
  author={Bian, Wei and Chen, Xiaojun},
  journal={Math. Program. publish online,},
  year={2024},
  publisher={Springer}
}

@article{jiang2023optimality,
  title={Optimality conditions for nonsmooth nonconvex-nonconcave min-max problems and generative adversarial networks},
  author={Jiang, Jie and Chen, Xiaojun},
  journal={SIAM J. Math. Data Sci.},
  volume={5},
  pages={693--722},
  year={2023},
  publisher={SIAM}
}

@article{shapiro2008stochastic,
  title={Stochastic mathematical programs with equilibrium constraints, modelling and sample average approximation},
  author={Shapiro, Alexander and Xu, Huifu},
  journal={Optimization},
  volume={57},
  pages={395--418},
  year={2008},
  publisher={Taylor \& Francis}
}

@book{dontchev2009implicit,
  title={Implicit Functions and Solution Mappings},
  author={Dontchev, Asen L and Rockafellar, R Tyrrell},
  volume={543},
  year={2009},
  publisher={Springer}
}

@article{izmailov2014strongly,
  title={Strongly regular nonsmooth generalized equations},
  author={Izmailov, Alexey F},
  journal={Math. Program.},
  volume={147},
  pages={581--590},
  year={2014},
  publisher={Springer}
}

@article{xu2010uniform,
  title={Uniform exponential convergence of sample average random functions under general sampling with applications in stochastic programming},
  author={Xu, Huifu},
  journal={J. Math. Anal.  Appl.},
  volume={368},
  pages={692--710},
  year={2010},
  publisher={Elsevier}
}

@inproceedings{jin2020local,
  title={What is local optimality in nonconvex-nonconcave minimax optimization?},
  author={Jin, Chi and Netrapalli, Praneeth and Jordan, Michael},
  booktitle={International Conference on Machine Learning},
  pages={4880--4889},
  year={2020},
  organization={PMLR}
}

@book{ioffe2017variational,
  title={Variational Analysis of Regular Mappings},
  author={Ioffe, Alexander D},
  journal={Springer Monographs in Mathematics. Springer, Cham},
  year={2017},
  publisher={Springer}
}

@article{cohen2025alternating,
  title={Alternating and parallel proximal gradient methods for nonsmooth, nonconvex minimax: a unified convergence analysis},
  author={Cohen, Eyal and Teboulle, Marc},
  journal={Math. Oper. Res.},
  volume={50},
  pages={141--168},
  year={2025},
  publisher={INFORMS}
}

@article{wachsmuth2013licq,
  title={On {LICQ} and the uniqueness of {L}agrange multipliers},
  author={Wachsmuth, Gerd},
  journal={Oper. Res. Lett.},
  volume={41},
  pages={78--80},
  year={2013},
  publisher={Elsevier}
}

@article{Sion58,
  title={On general minimax theorems},
  author={Sion, M.},
  journal={Pacific J. Math.},
  volume={8},
  pages={171--176},
  year={1958},
  publisher={}
}

@book{bonnans2013perturbation,
  title={Perturbation Analysis of Optimization Problems},
  author={Bonnans, J Fr{\'e}d{\'e}ric and Shapiro, Alexander},
  year={2013},
  publisher={Springer Science \& Business Media}
}

@article{TeboulleM18,
  title={A simplified view of first order methods for optimization},
  author={Teboulle, M},
  journal={Math. Program.},
  volume={170},
  pages={67--96},
  year={2018},
  publisher={Springer}
}

@article{rockafellar2019solving,
  title={Solving monotone stochastic variational inequalities and complementarity problems by progressive hedging},
  author={Rockafellar, R Tyrrell and Sun, Jie},
  journal={Math. Program.},
  volume={174},
  pages={453--471},
  year={2019},
  publisher={Springer}
}

@article{chen2019convergence,
  title={Convergence analysis of sample average approximation of two-stage stochastic generalized equations},
  author={Chen, Xiaojun and Shapiro, Alexander and Sun, Hailin},
  journal={SIAM J. Optim.},
  volume={29},
  pages={135--161},
  year={2019},
  publisher={SIAM}
}

@article{chen2019discrete,
  title={Discrete approximation of two-stage stochastic and distributionally robust linear complementarity problems},
  author={Chen, Xiaojun and Sun, Hailin and Xu, Huifu},
  journal={Math. Program.},
  volume={177},
  pages={255--289},
  year={2019},
  publisher={Springer}
}

@article{chen2022stochastic,
  title={Stochastic approximation methods for the two-stage stochastic linear complementarity problem},
  author={Chen, Lin and Liu, Yongchao and Yang, Xinmin and Zhang, Jin},
  journal={SIAM J. Optim.},
  volume={32},
  pages={2129--2155},
  year={2022},
  publisher={SIAM}
}
\end{document}